\documentclass[11pt]{amsart}

\usepackage{amsmath}
\usepackage{amssymb}
\usepackage{amsthm}
\usepackage{graphicx} 
\usepackage{url}
\usepackage{hyperref}
\usepackage{multirow}

\title[A template for the geodesic flow]{Coding of geodesics and Lorenz-like templates for some geodesic flows}

\author{Pierre Dehornoy}
\address{Univ.\ Grenoble Alpes, IF, F-38000 Grenoble, France; CNRS, IF, F-38000 Grenoble, France}
\email{pierre.dehornoy@ujf-grenoble.fr}
\urladdr{http://www-fourier.ujf-grenoble.fr/~dehornop/}

\author{Tali Pinsky}
\address{Department of Mathematics, University of British Columbia, Vancouver BC, Canada}
\email{tali@math.ubc.ca}
\urladdr{http://www.math.ubc.ca/~tali/}

\thanks{P. D. thanks \'Etienne Ghys, J\'er\^ome Los, and Vincent Pit for interesting discussions related to templates and geodesic flows, and the Technion for inviting him to Haifa in 2011 where this project was initiated. We thank the anonymous referee for several suggestions that improved the readability of the paper.}

\date{November 2014}

\theoremstyle{plain}
\newtheorem{introtheorem}{Theorem}

\theoremstyle{plain}
\newtheorem{theorem}{Theorem}[section]
\newtheorem{proposition}[theorem]{Proposition}
\newtheorem{lemma}[theorem]{Lemma}

\theoremstyle{definition}
\newtheorem{definition}[theorem]{Definition}

\newtheorem{question}[theorem]{Question}
\newtheorem{remark}[theorem]{Remark}

\newcommand{\Ad}{A_\downarrow}
\newcommand{\Ao}{A_0}

\newcommand{\Bd}{B_\downarrow}
\newcommand{\Bo}{B_0}
\newcommand{\bord}{\partial}
\newcommand{\bigon}{\beta}
\newcommand{\bouquet}{\mathcal{B}_{p,q,r}}
\newcommand{\bouquetHy}{\widetilde{\mathcal{B}_{p,q,r}}}

\newcommand{\Co}{C_0}
\newcommand{\Cd}{C_\downarrow}
\newcommand{\curvy}[1]{\stackrel{\thicksim}{#1}}
\newcommand{\cerA}{c_{A_B}}
\newcommand{\cerB}{c_{B_A}}
\newcommand{\cerAm}{c_{A_{B^-}}}
\newcommand{\cerBm}{c_{B_{A^-}}}
\newcommand{\Chain}{\mathcal{C}^3}

\newcommand{\deformation}{d_{p,q,r}}

\newcommand{\fgeod}{\varphi_{p,q,r}}
\newcommand{\ftemp}{\tau_{p,q,r}}
\newcommand{\fttemp}{\widetilde{\tau_{p,q,r}}}

\renewcommand{\ge}{\geqslant}
\newcommand{\Gpqr}{G_{p,q,r}}
\newcommand{\grpqr}{\widehat{\mathcal{G}_{p,q,r}}}

\newcommand{\Hopf}{\mathcal{H}^3_+}
\newcommand{\Hy}{\mathbb{H}^2}

\renewcommand{\le}{\leqslant}
\newcommand{\lex}{\le^{\mathrm{lex}}}

\newcommand{\N}{\mathbb{N}}

\newcommand{\PP}{P}

\newcommand{\sshift}{\hat\sigma}
\newcommand{\Sph}{\mathbb{S}}
\newcommand{\Spqr}{\Hy/\Gpqr}

\newcommand{\slex}{<^{\mathrm{lex}}}
\newcommand{\Spec}{S}

\newcommand{\template}{\mathcal T}
\newcommand{\Tpqr}{\template\!\!_{p,q,r}}
\newcommand{\tTpqr}{\widetilde{\template\!\!_{p,q,r}}}
\newcommand{\tdeformation}{\widetilde{d_{p,q,r}}}

\newcommand{\ut}{\mathrm{T}^1}
\newcommand{\utS}{\ut\Spqr}

\begin{document}

\maketitle

\begin{abstract}
We construct a template with two ribbons that describes the topology of all periodic orbits of the geodesic flow on the unit tangent bundle to any  sphere with three cone points with hyperbolic metric. 
The construction relies on the existence of a particular coding with two letters for the geodesics on these orbifolds.
\end{abstract}


\section{Introduction}
\label{S:Introduction}

For $p,q,r$ three positive integers---$r$ being possibly infinite---satisfying $\frac 1 p + \frac 1 q + \frac 1 r < 1$, we consider the associated hyperbolic triangle and the associated orientation preserving Fuchsian group~$\Gpqr$. 
The quotient $\Spqr$ is a sphere with three cone points of angles~$\frac {2\pi} p, \frac {2\pi} q, \frac {2\pi} r$ obtained by gluing two triangles. 
The unit tangent bundle~$\utS$ is a 3-manifold that is a Seifert fibered space. 
It naturally supports a flow whose orbits are lifts of geodesics on~$\Spqr$. 
It is called the \emph{geodesic flow on~$\utS$} and is denoted by~$\fgeod$.
These flows are of Anosov type~\cite{Anosov} and, as such, are important for at least two reasons: 
they are among the simpliest chaotic systems~\cite{Hadamard} and they are fundamental objects in 3-dimensional topology~\cite{Thu88}.
Each of these flows has infinitely many periodic orbits, which are all pairwise non-isotopic. 
The study of the topology of these periodic orbits began with David Fried who showed that many collections of such periodic orbits form fibered links~\cite{Fried}. 
More recently \'Etienne Ghys gave a complete description in the particular case of the modular surface---which corresponds to $p=2, q=3, r=\infty$---by showing that the periodic orbits are in one-to-one correspondence with periodic orbits of the Lorenz template~\cite{Ghys}.

The goal of this paper is to extend this result by giving an explicit description of the isotopy classes of all periodic orbits of~$\fgeod$ for every~$p,q,r$ (with $p,q$ finite). 
A \emph{template}~\cite{BW} is an embedded branched surface made of several ribbons and equipped with a semi-flow. 
A template is characterized by its embedding in the ambient manifold and by the way its ribbons are glued, namely by the \emph{kneading sequences} that describe the left- and rightmost orbit of every ribbon \cite{HS,dMvS}. 
Call~$\Tpqr$ the template with two ribbons whose embedding in~$\utS$ is depicted on Figure~\ref{F:Tpqr} 
and whose kneading sequences are the words $u_L, u_R, v_L, v_R$ given by Table~\ref{Table}, where the letter~$a$ corresponds to travelling along the left ribbon and $b$ to travelling along the right ribbon.
If $r$ is infinite, the template is open of both sides of every ribbon, so that the kneading sequences are not realized by any orbit of the template. 
If $r$ is finite, both ribbons are closed on the left and open on the right, so that $u_L$ and $v_L$ are realized, but not $u_R$ nor~$v_R$.

\begin{introtheorem}
  	\label{T:Template}
  	Up to one exception when $r$ is finite, there is a one-to-one correspondance between periodic orbits of the geodesic flow~$\fgeod$ on~$\utS$ and periodic orbits of the template~$\Tpqr$ such that its restriction to every finite collection is an isotopy.
	The exception consists of the two orbits of~$\Tpqr$ whose codes~$w_L$ and $w_R$ are given by Table~2 and which actually correspond to the same orbit of~$\fgeod$.
\end{introtheorem}

\begin{figure}[ht]
  	\begin{picture}(90,70)(0,0)
      		\put(0,0){\includegraphics*[width=.6\textwidth]{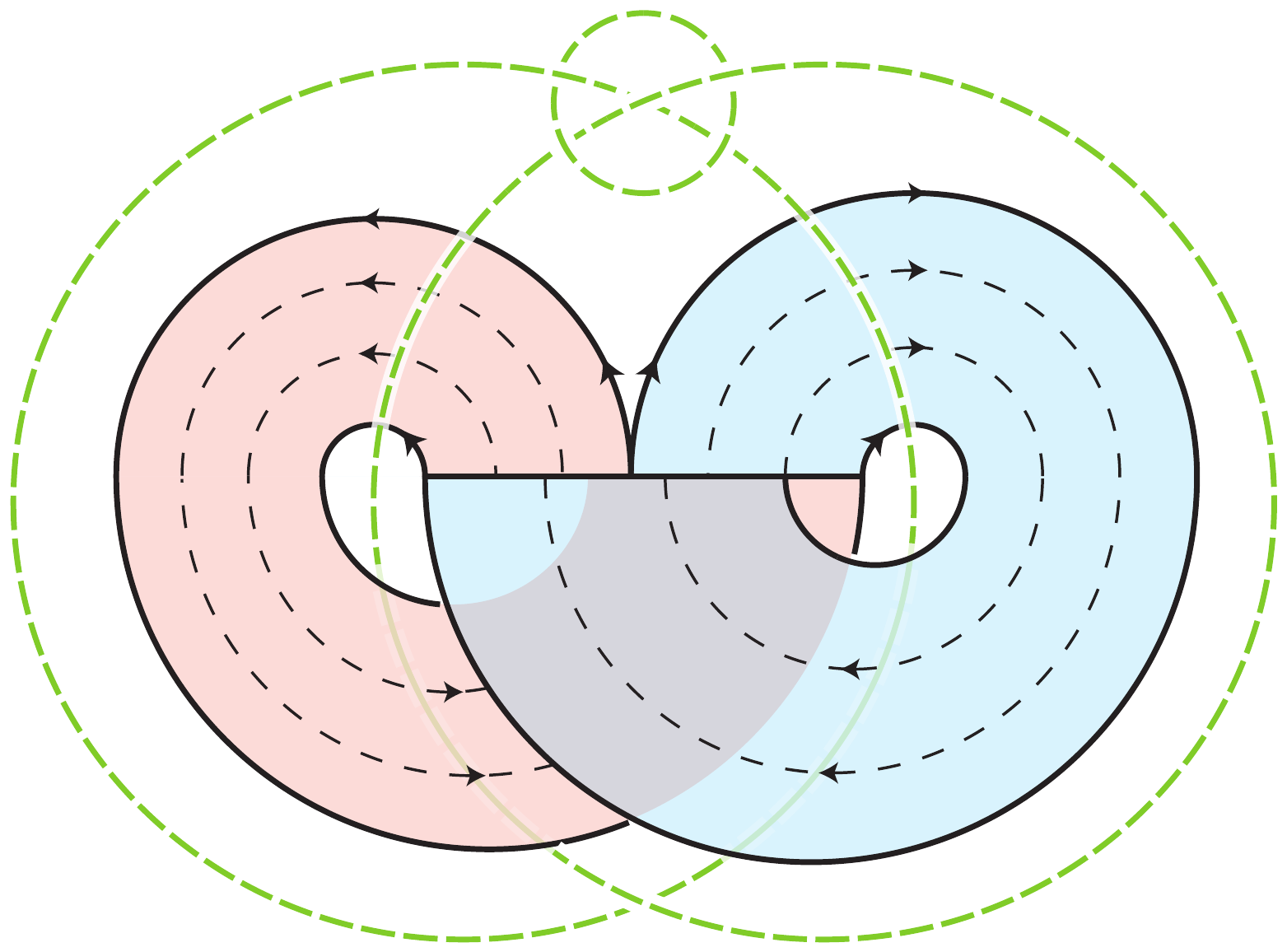}}
      		\put(-4,45){$p{-}1$}
      		\put(87,45){$q{-}1$}
      		\put(42,67){$r{-}1$}
      		\put(18.6,35.5){$u_L$}
      		\put(67,35.5){$v_R$}
      		\put(47.5,42){$v_L$}
      		\put(36.5,42){$u_R$}
  	\end{picture}
  	\caption{\small The template~$\Tpqr$ in~$\utS$. 
  	The 3-manifold~$\utS$ is obtained from~$\Sph^3$ by surgeries on a three component Hopf link (dotted) with the given indices. 
  	The template~$\Tpqr$ is characterized by its embedding in~$\utS$ and by the kneading sequences that describe the orbits of the extremities of the ribbons. 
  	When $r$ is infinite, $\utS$ is the open manifold obtained by removing the component labelled~$r-1$ of the 3-component Hopf link.
	When $p=2$, the exit side of the left ribbon is strictly included into the entrance side of the right ribbon.}
  	\label{F:Tpqr}
\end{figure}

\begin{table}[htb]
	\begin{center}
	\begin{tabular}{c|c|c}
		& $u_L$ & $v_R$ \\ \hline
   		$r$ infinite & $(a^{p-1}b)^\infty$  
  	  	& $(b^{q-1}a)^\infty$ \\ \hline
		\hline
		$p, q, r>2$& & \\ \hline
  		$r$ odd & $((a^{p-1}b)^{\frac{r-3}{2}} a^{p-1}b^2)^\infty$  
  	  	& $((b^{q-1}a)^{\frac{r-3}{2}} b^{q-1}a^2)^\infty$ \\ \hline
		$r$ even & $((a^{p-1}b)^{\frac{r-2}{2}} a^{p-2}(ba^{p-1})^{\frac{r-2}{2}}b^2)^\infty$
		& $((b^{q-1}a)^{\frac{r-2}{2}} b^{q-2}(ab^{q-1})^{\frac{r-2}{2}}a^2)^\infty$ \\ \hline
		\hline
		$p=2, q>2,r>4$& & \\ \hline
   		$r$ odd & $((ab)^{\frac{r-3}{2}} ab^2)^\infty$  
  	  	& $b^{q-1}((ab^{q-1})^{\frac{r-5}{2}} ab^{q-2})^\infty$ \\ \hline
  		$r$ even & $((ab)^{\frac{r-4}{2}} ab^2)^\infty$  
  	  	& $b^{q-1}((ab^{q-1})^{\frac{r-4}{2}} ab^{q-2})^\infty$ \\ \hline
	\end{tabular}

	\vspace{1mm}
	 $v_L:=bu_L,$ 
	 
	 $u_R:=av_R$
	 \end{center} 
	\caption{\small The kneading sequences of the template~$\Tpqr$. }
	\label{Table}
\end{table}

\begin{table}[htb]
	\begin{center}
	\begin{tabular}{c|c|c}
		$p,q,r>2$& $w_L$ & $w_R$ \\ \hline
   		$r$ odd & $((a^{p-1}b)^{\frac{r-3}{2}} a^{p-2}b)^\infty$  
  	  	& $((b^{q-1}a)^{\frac{r-3}{2}} b^{q-2}a)^\infty$ \\ \hline
		$r$ even & $((a^{p-1}b)^{\frac{r-2}{2}} a^{p-2}b(a^{p-1}b)^{\frac{r-4}{2}}a^{p-2}b)^\infty$
		& $((b^{q-1}a)^{\frac{r-2}{2}} b^{q-2}a(b^{q-1}a)^{\frac{r-4}{2}}b^{q-2}a)^\infty$ \\ \hline
		\hline
		$p=2,q>2,r>4$& & \\ \hline
   		$r$ odd & $((ab)^{\frac{r-3}{2}} ab^2)^\infty$  
  	  	& $((ab^{q-1})^{\frac{r-5}{2}} ab^{q-2})^\infty$ \\ \hline
		$r$ even & $((ab)^{\frac{r-4}{2}} ab^2)^\infty$
		& $((ab^{q-1})^{\frac{r-4}{2}} ab^{q-2})^\infty$ \\ \hline
	\end{tabular}
	\end{center}
	\caption{\small The codes of the two periodic orbits of~$\Tpqr$ that correspond to the same periodic orbit of the geodesic flow~$\fgeod$.}
	\label{TableExc}
\end{table}

In subsequent works, we intend to use the template~$\Tpqr$ for deriving important properties of the geodesic flow~$\fgeod$, in particular the fact that all periodic orbits form prime knots, or the left-handedness (a notion introduced in~\cite{GhysLeftHanded}) of~$\fgeod$ (see~\cite{PierreLeftHanded}).

The kneading sequences (Table~\ref{Table}) may look complicated. 
In a sense, it is the cost for having a template with two ribbons only. 
The proof suggests that, in the case of $r$ finite, there are infinitely many other possible words, thus leading to other templates (the difference lying in the kneading sequences, not in the embedding). 
But these words are not simpler than the ones we propose here. 

Theorem~\ref{T:Template} is to be compared with previous works by Ghys~\cite{Ghys}, Pinsky~\cite{Tali}, and Dehornoy~\cite{Pierre}. 
The results of~\cite{Ghys, Tali} deal with orbifolds of type $\Hy/G_{2,3,\infty}$ and~$\Hy/G_{2,q,\infty}$ for all $q\ge 3$ respectively. 
The existence of a cusp in these cases allows the obtained templates to have trivial kneading sequences of the form~$a^\infty$ and~$b^\infty$. 
These constructions can be recovered from Theorem~\ref{T:Template}. 
The construction in~\cite{Pierre} deals with compact orbifolds, including the case of~$\Spqr$, but the constructed templates have numerous ribbons instead of two here, and the kneading sequences are not determined. 
Thus, Theorem~\ref{T:Template} is a generalization of the results of~\cite{Ghys, Tali} to compact orbifold of type~$\Spqr$, as well as a simplification of the construction of~\cite{Pierre} with more precise information.

The main idea for the proof of Theorem~\ref{T:Template} is to distort all geodesics in~$\Spqr$ onto an embedded graph that is just a bouquet of two oriented circles. 
We perform the distortion in a equivariant way in $\Hy$ (or in fact in $T^{1} \Hy$), and this allows one to deform through the cone points. We use this freedom to fix a direction in which the deformed geodesic is allowed to wind around a cone point, and this allows us to choose such a simple graph and a template with only two ribbons.
In order to have a one-to-one correspondence between orbits of the flow and orbits of the template, the distortion has to obey some constraints. 
We introduce the notion of \emph{spectacles}, which allows one to verify that we always get at most one coding for any semi-infinite geodesic
and that the coding is nice. 
Ultimately we introduce what we call an \emph{accurate pairs of spectacles}, that ensures we do get some coding for any bi-infinite geodesic.
The main intermediate step is the following result which might be of independent interest (see the introduction of Section~\ref{S:Coding} for the definitions of the graph~$\grpqr$, of the code, and of the supershift operator~$\sshift$)

\begin{introtheorem}
	\label{P:Coding}
	Assume that $\eta, \xi$ in~$\bord\Hy$ are the two extremities of a lift in~$\Hy$ of a periodic geodesic in~$\Spqr$, then there exists a bi-infinite path~$\gamma$ in~$\grpqr$ that connects~$\eta$ to~$\xi$ and whose code $w(\gamma)$ satisfies $u_L\le \sshift^k(w(\gamma))< u_R$ (\emph{resp.} $v_L\le \sshift^k(w(\gamma))< v_R$) for every shift $\sshift^k(w(\gamma))$ that begins with a power of the letter~$a$ (\emph{resp.} $b$). 
This path is unique, except for the paths encoded by the words~$w_L$ and $w_R$ that have the same extremities. 
\end{introtheorem}

Theorem~\ref{P:Coding} is reminiscent of results by Caroline Series~\cite[Thm 2.7]{Series81} and J\'er\^ome Los~\cite[Thm 4.4]{Los}. 
The novelty here is that we deal with monoids instead of groups, and that we cover some cases that were not addressed by Series (for small values of $p, q,$ and $r$).

The article is organized as follows.
First we recall in Section~\ref{S:US} the topology of the space~$\utS$ by giving a surgery presentation on a link in~$\Sph^3$.
We then describe in Section~\ref{S:Coding} a coding with two letters only for geodesics on~$\Spqr$. 
In Section~\ref{S:Tpqr}, we use the coding for constructing a template with two ribbons in~$\utS$ and describe its kneading sequences in terms of the chosen coding. 
We conclude the article with a few questions in Section~\ref{S:Questions}.
 
In the whole text, $p,q,r$ denote three natural numbers (with $r$ possibly infinite) satisfying $\frac 1 p+\frac1 q+\frac1 r<1$. 
We denote by $\Ao\Bo\Co$ a fixed hyperbolic triangle in $\Hy$ with respective angles $\pi/p, \pi/q$, and~$\pi/r$.
We denote by $\Gpqr$ the orientation-preserving Fuchsian group generated by the rotations of angles $2\pi/p, 2\pi/q$, and~$2\pi/r$ around $\Ao, \Bo,$ and~$\Co$. 
The quotient $\Spqr$ is then a 2-dimensional orbifold, namely a sphere with three cone points (or a cusp when $r$ is infinite), that we call $\Ad, \Bd, \Cd$.
The letters $A$ and $B$ with sub/superscripts will always denote points of$~\Hy$ in the~$\Gpqr$-orbit of~$\Ao$ and $\Bo$ respectively.


\section{The topology of~$\utS$}
\label{S:US}

The unit tangent bundle~$\utS$ is a Seifert fibered space whose regular fibers correspond to fibers of regular points of~$\Spqr$ and exceptional fibers to the fibers of the cone points $\Ad, \Bd, \Cd$. 
The goal of this section (Proposition~\ref{P:Surgery}) is to give a surgery presentation of the 3-manifold~$\utS$ with explicit coordinates. 
This has already been done, for example in~\cite[p.\ 183]{Montesinos}. 
The presentation we give here is slightly different, but more suited to the description of the template~$\Tpqr$. 
Using Kirby calculus (see~\cite[p.\ 267]{Rolfsen}), one can check that our presentation yields the same manifold as Montesinos'.
In this section, the numbers $p,q,r$ can be finite or infinite.

We begin with some elementary lemmas that help us fixing some notation. 
Assume that $\gamma$ is an oriented closed curve immersed in a surface.
Then the restriction~$\ut\gamma$ to the fibers of the points of~$\gamma$ of the unit tangent bundle to the surface is a 2-torus which supports four particular homology classes, represented by four vector fields: 

1. the fiber~$f$ of a given point of~$\gamma$, oriented trigonometrically,

2. an index~$0$ vector field~$z_\gamma$; for example, a vector field that is everywhere tangent to~$\gamma$,

3. an index~$+1$ vector field~$u_\gamma^+$, that is, a vector field that rotates once to the left when traveling along~$\gamma$,

4. an index~$-1$ vector field~$u_\gamma^-$, that is, a vector field that rotates once to the right when traveling along~$\gamma$.

\begin{lemma}
	\label{L:Basis}
	For $\gamma$ an embedded oriented curve, the classes $[f]$ and $[z_\gamma]$ form a basis of~$\ut\gamma$. 
	In this basis, one has $[u_\gamma^-] = -[f]+[z_\gamma]$ and $[u_\gamma^+] = [f]+[z_\gamma]$. 
\end{lemma}

\begin{proof}
	The vector fields $f$ and $z_\gamma$ have exactly one point in common where they intersect transversally, so they form a basis of~$H_1(\ut\gamma)$. 
	Rotating along a fiber amounts to turning to the left, so one get $[u_\gamma^+] = [f]+[z_\gamma]$. 
	Similarly one obtains $[u_\gamma^-] = -[f]+[z_\gamma]$.
\end{proof}

Let $\PP$ be an oriented pair of pants whose three oriented boundary components are denoted by~$\gamma_A, \gamma_B, \gamma_C$. 
Consider the vector field~$v$ on~$\PP$ given by Figure~\ref{F:PairOfPants} and denote by~$v_A, v_B, v_C$ its respective restrictions to~$\gamma_A, \gamma_B, \gamma_C$. 
Since $v_A$ is tangent to~$\gamma_A$, one has $[v_A] = [z_A]$ on~$\ut\gamma_A$. 
Similarly one has~$[v_B] = [z_B]$ on~$\ut\gamma_B$. 
Now $v_C$ is not tangent to~$\gamma_C$. Since it rotates plus one times when traveling along~$\gamma_C$, one has~$[v_C] = [u^+_C]$ on~$\ut\gamma_C$.

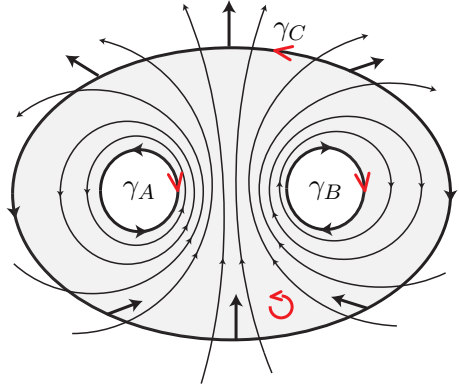
\begin{figure}[ht]
   	\begin{picture}(60,50)(0,0)
      	\put(0,-3){\includegraphics*[width=.4\textwidth]{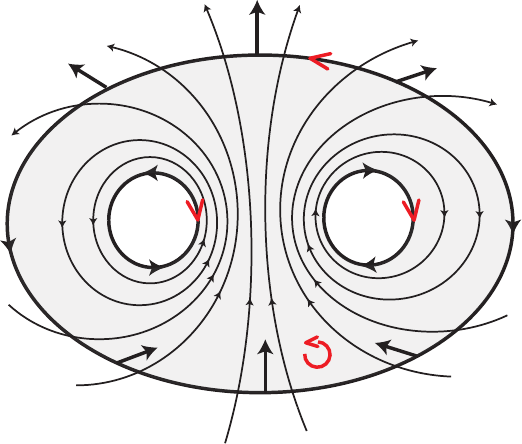}}
      	\put(15.5,22.2){$\gamma_A$}
      	\put(40.2,22.2){$\gamma_B$}
      	\put(35.5,43.4){$\gamma_C$}
   	\end{picture}
   	\caption{\small A pair of pants~$\PP$ and the vector field~$v$. 
	The orientations of the boundary components induced by the orientation of~$\PP$ are with larger red arrows.}
	\label{F:PairOfPants}
\end{figure} 

We denote by $\Chain=U^A\cup U^B \cup U^C$ the 3-component link in~$\Sph^3$ that is a chain of 3 unknots, with $U^C$ in the middle (see Figure~\ref{F:3Chain}).

\begin{figure}[ht]
   	\begin{picture}(120,65)(0,0)
      	\put(0,0){\includegraphics*[width=.8\textwidth]{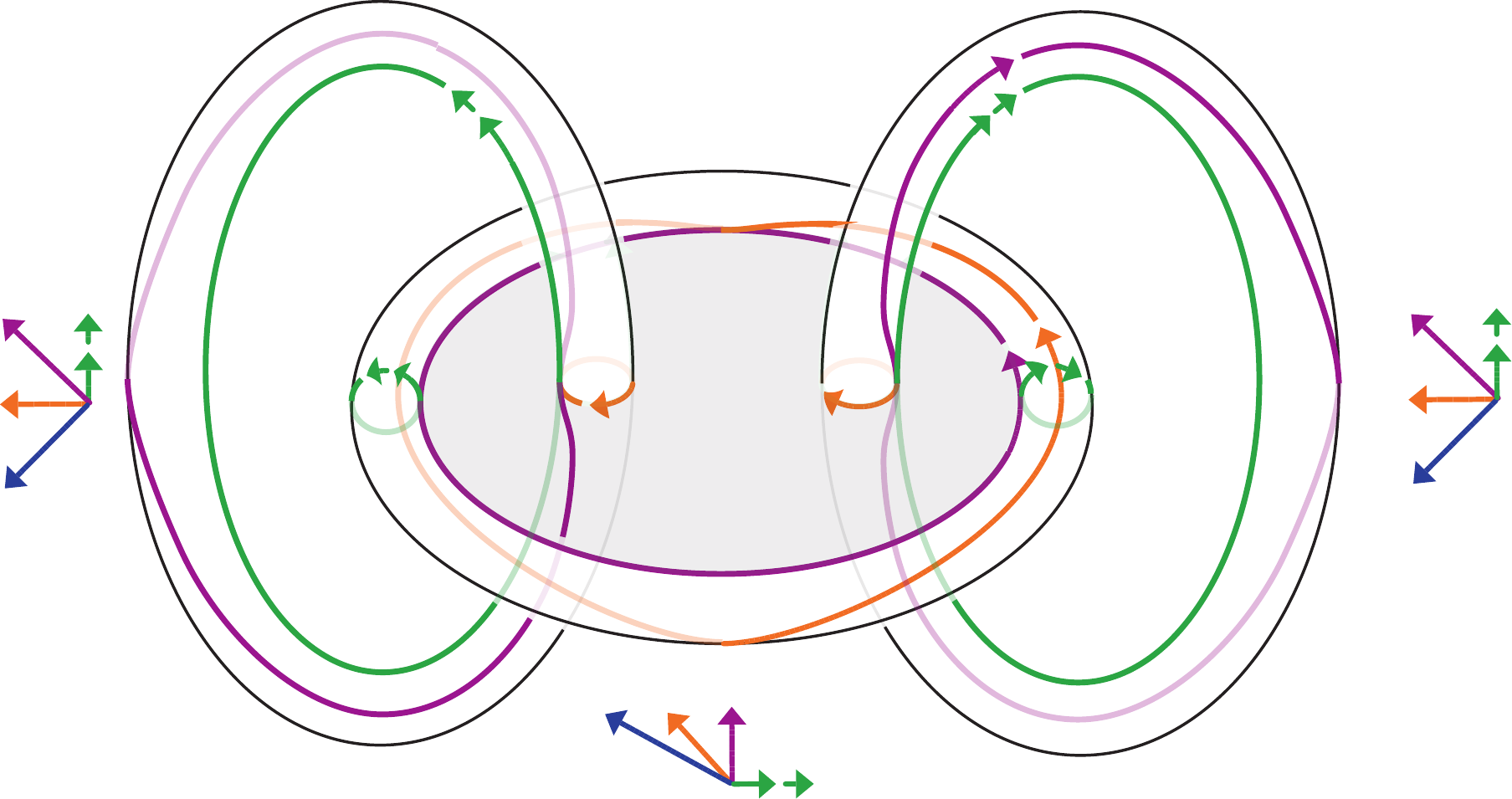}}
	\put(16,60){$U^A$}
	\put(96,60){$U^B$}
	\put(56,51){$U^C$}
	\put(65,1){$f_C$}
	\put(57,8){$u^+_C$}
	\put(51,8){$z_C$}
	\put(43,6){$u^-_C$}
	\put(5,40){$f_A$}
	\put(-3,40){$u^+_A$}
	\put(-4,31){$z_A$}
	\put(-3,22){$u^-_A$}
	\put(117,40){$f_B$}
	\put(107.5,40){$u^+_B$}
	\put(107.5,31){$z_B$}
	\put(108,22){$u^-_B$}
    	\end{picture}
   	\caption{\small The unit tangent bundle to a pair of pants~$\ut \PP$ seen as a product~$\PP\times\Sph^1$, where the section corresponding to the vector field~$v$ on~$\PP$ of Figure~\ref{F:PairOfPants} is shaded. The vector fields~$f$ (double arrow, green), $z$ (orange) and $u^+$ (purple) for the three boundary tori~$\ut\gamma_A, \ut\gamma_B$ and $\ut\gamma_C$ are represented. Next to every torus is represented the homology classes of these vector fields and of $u^-$ in the $\{$meridian, longitude$\}$-basis.}
	\label{F:3Chain}
\end{figure}
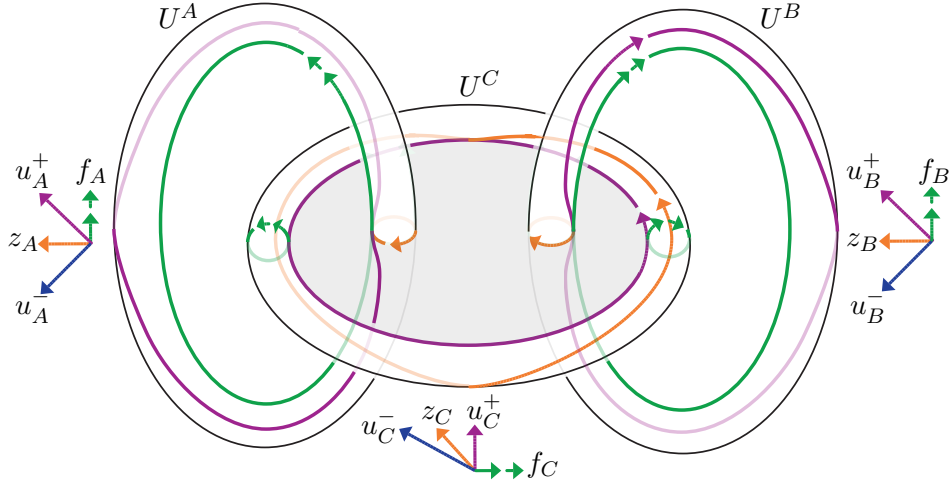 

\begin{lemma}
	\label{L:3Chain}
	The unit tangent bundle to the pair of pants~$\PP$ is homeomorphic to~$\Sph^3\setminus\Chain$, where the respective meridians of the components~$U^A, U^B, U^C$ of~$\Chain$ are represented by the vector fields~$-z_A, -z_B, f_C$  and the respective longitudes by~$f_A, f_B, u^+_C$.
\end{lemma}

\begin{proof}
	Since $\PP$ is an open surface, the unit tangent bundle~$\ut \PP$ is homeomorphic to the product~$\PP\times\Sph^1$, that is, to~$\Sph^3\setminus~\Chain$ (see Figure~\ref{F:3Chain}). 
	The vector field~$v$ given by Figure~\ref{F:PairOfPants} yields a section for this product whose boundary consists of two meridians of~$U^A$ and~$U^B$ respectively and one longitude of~$U^C$. 
	The respective longitudes of~$U^A, U^B$ and the meridian of~$U^C$ correspond to three fibers (as can be checked on Figure~\ref{F:3Chain}). 
\end{proof}

The above presentation of~$\ut \PP$ is not symmetric in the three boundary components. 
This symmetry can be recovered using a twist.
Denote by~$\Hopf = H^A\cup H^B \cup H^C$ the positive 3-component Hopf link in~$\Sph^3$ (the green link in Figure~\ref{F:Tpqr}).

\begin{lemma}
\label{L:SurgeryInf}
	The unit tangent bundle to the pair of pants~$\PP$ is homeomorphic to~$\Sph^3\setminus\Hopf$, where the respective meridians of the components~$H^A, H^B, H^C$ of~$\Chain$ are represented by the vector fields~$-z_A, -z_B, -z_C$  and the respective longitudes by~$u^+_A, u^+_B, u^+_C$. 
\end{lemma}

\begin{figure}[ht]
   	\begin{picture}(135,65)(0,0)
      	\put(0,0){\includegraphics*[width=.9\textwidth]{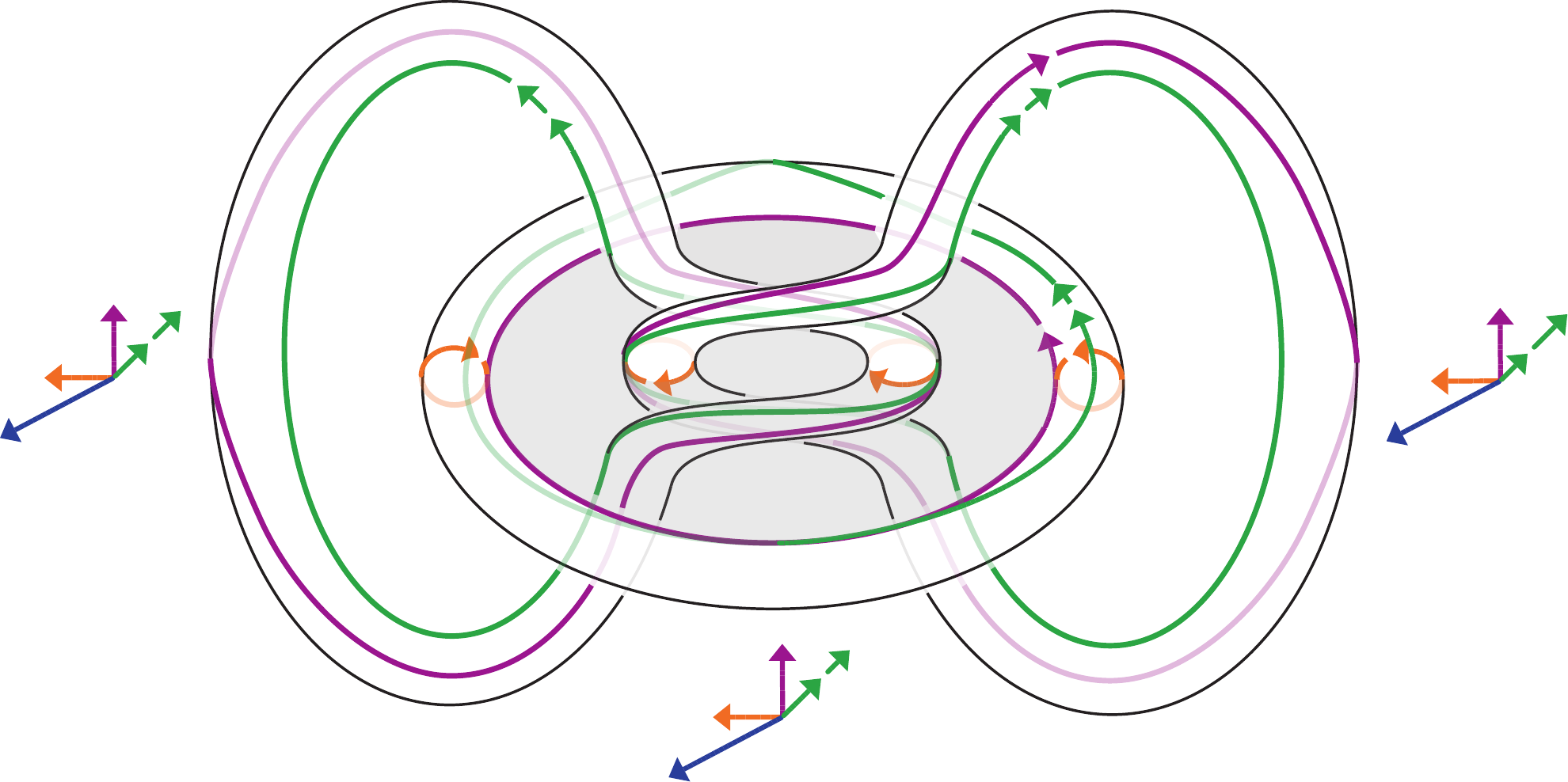}}
	\put(20,60){$H^A$}
	\put(110,60){$H^B$}
	\put(65,55){$H^C$}
	\put(13,42){$f_A$}
	\put(6,43){$u^+_A$}
	\put(-0.5,35){$z_A$}
	\put(-5,29){$u^-_A$}
	\put(134,42){$f_B$}
	\put(126,42.5){$u^+_B$}
	\put(119,35){$z_B$}
	\put(118,26){$u^-_B$}
	\put(74,10){$f_C$}
	\put(61,11){$u^+_C$}
	\put(57,6){$z_C$}
	\put(52,0){$u^-_C$}
    	\end{picture}
   	\caption{\small The unit tangent bundle to a pair of pants~$\ut \PP$ as the complement of a Hopf link. 
	It is obtained from the previous one by a twist along the shaded disc. 
	The vector fields~$f$ (double arrow, green), $z$ (orange) and $u^+$ (purple) for the three boundary tori~$\ut\gamma_A, \ut\gamma_B$ and $\ut\gamma_C$. 
	On all three tori are represented the homology classes of these vector fields and $u^-$, in the $\{$meridian, longitude$\}$-basis. 
	All meridians are of type~$-z$ and all longitudes of type~$u^+$.}
	\label{F:Hopf}
\end{figure}
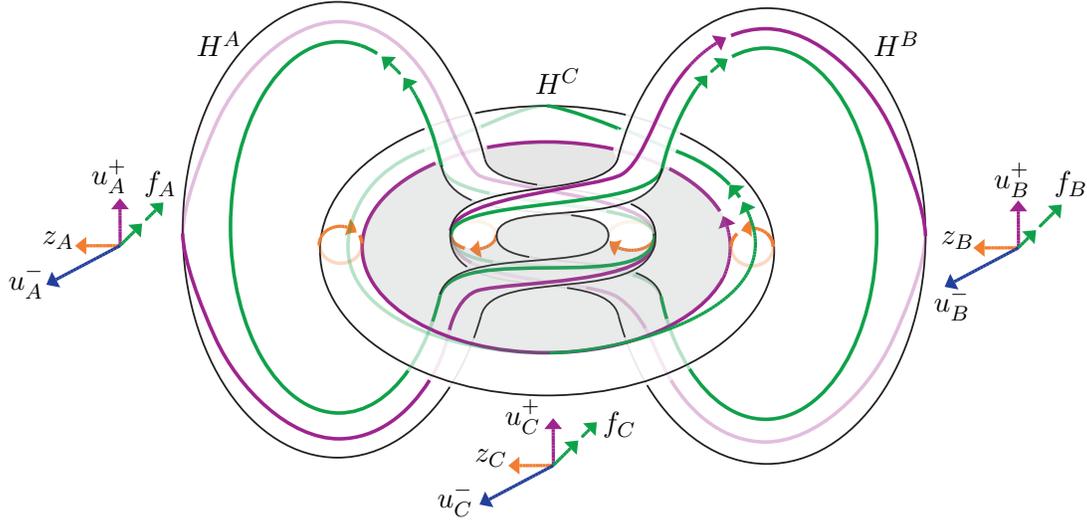 

\begin{proof}
	Starting from the presentation~$\ut \PP\simeq\Sph^3\setminus\Chain$ of Lemma~\ref{L:3Chain}, we perform a positive twist~$\tau^C$ on the component~$U^C$, 
	that is, we cut~$\Sph^3$ along a disc bounded by~$U^C$ and we glue back after performing one full positive turn along this disc (see Figure~\ref{F:Hopf}).
	We denote by~$H^A, H^B, H^C$ the images of~$U^A, U^B, U^C$ under~$\tau^C$.
	The tori~$H^A, H^B$ as well as the fibers of the points of~$\PP$ all get linked plus one times. 
	On these two tori, the twist performs a positive transvection, adding a meridian to the longitude. 
	Since the meridian of~$U^A$ is represented by the vector field $-z^A$, the meridian of~$H^A$ is also represented by~$-z^A$ and the longitude by~$f^A - (-z^A) = u_+^A$. 
	The same holds on~$H^B$.
	On the torus~$U^C$, the twist performs another transvection, adding a longitude to the meridian. 
	Therefore the longitude of~$H^C$ is represented by~$u_+^C$, as is the longitude of~$U^C$. 
	Its meridian is represented by~$f^C - u_+^C = -z^C$.  
\end{proof}

In the above presentation of~$\ut \PP$ as $\Sph^3\setminus\Hopf$, the fibers of all points of~$\PP$ are fibers of the positive Hopf fibration (see~\cite{Dimensions} for explanations and movies). 

We can now give the desired presentation of~$\utS$. 
Recall that a Dehn filling with slope~$a/b$ on a knot amounts to gluing in a solid torus so that a meridian of the solid torus intersects $a$ longitudes and $b$ meridians of the knot.
With this convention, a Dehn filling with slope $\infty$ amounts to filling the removed knot back in.
Here we add the extra-convention that in that case of infinite slope, we keep the knot removed.

\begin{proposition}
  	\label{P:Surgery}
  	The 3-manifold~$\utS$ can be obtained from $\Sph^3\setminus\Hopf$ by Dehn filling the three components of~$\Hopf$ with respective slopes $p-1,q-1$, and~$r-1$. 
  	Moreover, the fibers of~$\utS$ correspond to the fibers of the Hopf fibration of~$\Sph^3\setminus\Hopf$ and the exceptional fibers to the cores of the Dehn fillings.
\end{proposition}

\begin{proof}
	Since $\Hy/G_{\infty, \infty, \infty}$ is a pair of pants, Lemma~\ref{L:SurgeryInf} gives the result in this case.
	
	We now suppose $p$ finite (and $q,r$ arbitrary). 
	Let $D_p$ be the quotient of a disc~$D$, by a rotation of order~$p$.
	The orbifold~$\Hy/G_{p, q, r}$ is obtained from $\Hy/G_{\infty, q,r}$ by gluing $\ut D_p$ into the cusp~$\gamma_A$ (corresponding to the fiber of~$\Ad$), formally by first removing a neighbourhood of~$\gamma_A$.
	 
	The unit tangent bundle $\ut D_p$ is a solid torus whose meridian disc corresponds to the image on~$D_p$ 
	of a non-singular vector field on~$D$ (see Figure~\ref{F:Meridian}).
	Such a vector field intersects each non singular fiber $p$ times and intersects a vector field tangent to~$\bord D_p$ once (Figure~\ref{F:Meridian} $(d,e)$). 
	Therefore after gluing, the meridian of $\ut D_p$ is glued to a curve with coordinates~$(p-1,\pm 1)$ in the $\{$meridian, longitude$\}$-basis of~$H_A$.
	The sign can be fixed either by working out explicitly the action of $S^1$ as in \cite{Montesinos}, or simply by checking that in the $p=1$ case the meridian is glued to the curve~$u^+$ as depicted in Figure~\ref{F:Meridian} $(a)$, hence has coordinates $(0,1)$.
	Therefore the gluing corresponds to a Dehn filling of slope~$p-1$.

\begin{figure}[ht]
	\begin{picture}(145,75)(0,0)
	\put(0,0){\includegraphics*[width=\textwidth]{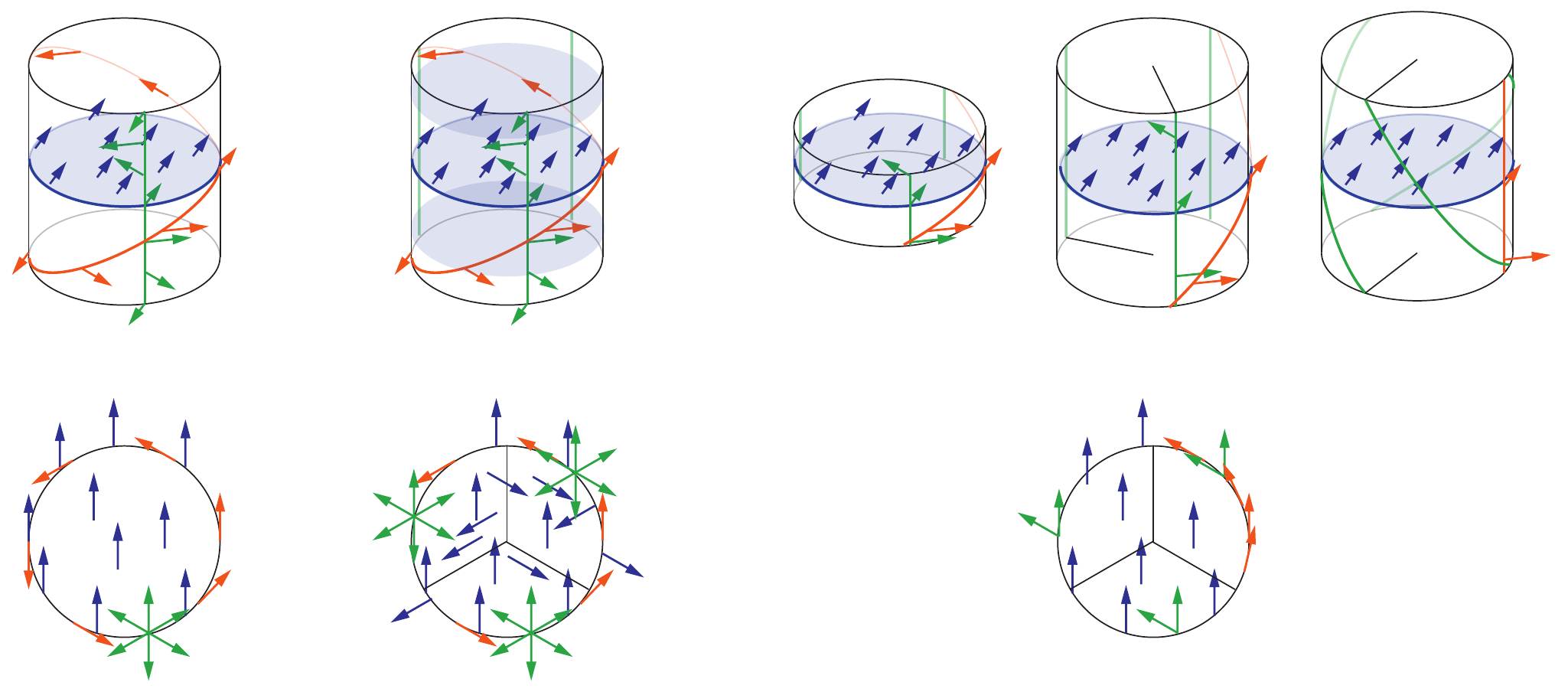}}
	\put(10,68){$(a)$}
	\put(47,68){$(b)$}
	\put(84,68){$(c)$}
	\put(108,68){$(d)$}
	\put(133,68){$(e)$}
	\end{picture}
	\caption{\small $(a)$ The unit tangent bundle~$\ut D$ to a disc~$D$ is a solid torus~$D\times\Sph^1$. 
	The fiber~$\ut\{*\}=f_*$ of a point~$*$ of~$\bord D$ is shown in green, a vector field~$z_{\bord D}$ tangent to~$\bord D$ in orange and a meridian disc given by a constant vector field in blue. 
	When traveling along~$\bord D$ with the orientation induced by the exterior (hence in the clockwise direction), 
	the blue vector field has index~$+1$, that is, it rotates to the left. 
	It is then isotopic to~$u^+_{\bord D}$. 
	$(b)$ the action of an order~$3$ rotation on~$\ut D$: the vector field~$z_{\bord D}$ is invariant, 
	while the meridian disc (in blue) is rotated in the fiber direction, and the fiber~$f_*$ (in green) is rotated around~$D$. 
	$(c)$ the quotient of~$\ut D$ by an order~$3$ rotation: 
	a fundamental domain is given by a horizontal slice of the full torus made of those tangent vectors that point between 9 o'clock and 1 o'clock, say. 
	$(d)$ the same picture expanded in the fiber direction. The top and bottom faces are identified using an order~$3$ rotation. 
	$(e)$ after a twist of $-1/3$ of a turn, $\ut D_3$ is a solid torus with the standard identification of the opposite faces of a cylinder. 
	The meridian disc is still given by the blue vector field. 
	It intersects the fiber of a point (in green) 3 times , and the vector field tangent to~$\bord D_3$ (in orange) minus one times.}
	\label{F:Meridian}
\end{figure}
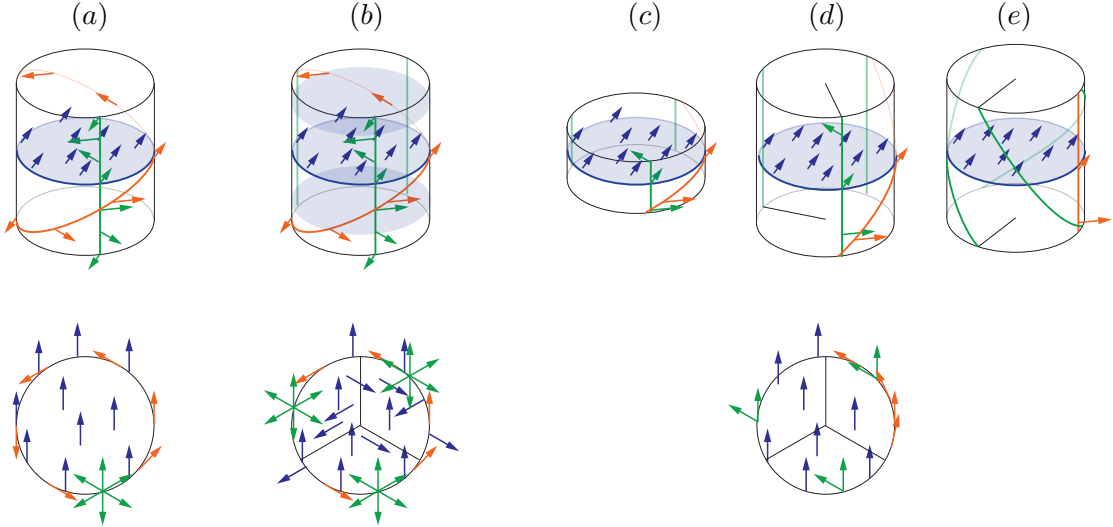

	The cases of $q$ and $r$ are treated similarly.
\end{proof}

Note that, if we had not perform the negative Dehn twist in Lemma~\ref{L:SurgeryInf}, we would obtain the complement of a chain of three unknots instead of the Hopf link, as depicted on Figure~\ref{F:3Chain}.
In this case, a similar statement to Proposition~\ref{P:Surgery} holds by replacing the surgery coeffecients by $p, q$, and $1{-}1/r$ respectively (again, the change of coefficients is coherent with Rolfsen's move~\cite[p.267]{Rolfsen}).


\section{Coding of the geodesic flow}
	\label{S:Coding}

The geodesic flow~$\phi$ on~$\ut\Hy\simeq \Hy\times\Sph^1$ is defined in the following way: 
every unit tangent vector to~$\Hy$ is of the form~$(\gamma(0), \dot\gamma(0))$ where $\gamma$ is a geodesic traveled at speed~$1$, we then set $\phi^t(\gamma(0), \dot\gamma(0)) = (\gamma(t), \dot\gamma(t))$. 
For $G$ a Fuchsian group, this definition is $G$-equivariant, and the geodesic flow then projects on~$\ut\Hy/G$.

These flows are the oldest known example of Anosov flows~\cite{Anosov}.
Their hyperbolic character implies the existence of a Markov partition~\cite{Ratner}, that is, of a decomposition of the flow into flow boxes whose entry and exit faces glue nicely. 
However, describing such a coding for explicit groups (for example for surface or triangular groups) is not so easy. 
This problem has a long history, see among others~\cite{MH,BS, Series81, AF, CN, Katok, Pit} and the references therein. 

Here we adapt to our needs the coding given by Caroline Series~\cite{Series81} and completed by J\'er\^ome Los~\cite{Los} (who covers some cases that were left out by Series).
The main idea in these codings is to distort geodesics in~$\Hy$ on a planar Cayley graph of~$G$.
Here we do not even need a Cayley graph, but only a planar graph that is $G$-invariant. 
In this context, all properties of Series-Los' coding still hold.
Our present generalisation works for general Fuchsian groups, but since stating the construction in full generality would make the notations heavier, we only state it in the case we are interested in.

We start with the triangle $\Ao\Bo\Co$ in~$\Hy$ and the group~$\Gpqr$ as before. 
We denote by~$\grpqr$ the embedded graph in~$\Hy$ whose vertices are the images of~$\Ao$ and $\Bo$ by~$\Gpqr$, and whose edges are the images of the segment~$(\Ao\Bo)$ by~$\Gpqr$ (see Figure~\ref{F:Grpqr}).
All vertices of~$\grpqr$ have degree~$p$ or~$q$. 
All components of the complement~$\Hy\setminus\grpqr$ are polygons with $2r$ vertices, half of which are in the $\Gpqr$-orbit of~$A_0$ and half of which in the orbit of~$B_0$.

\begin{definition}[see Figure~\ref{F:Grpqr}]
 	\label{D:coding}
  	Assume that $\gamma:=(B^0, A^1, B^2, \dots)$ is a semi-infinite path in~$\grpqr$.
  	The \emph{code} of~$\gamma$ is the semi-infinite word~$w(\gamma)$ defined as~$a^{k_1}b^{k_2}a^{k_3}b^{k_4}\dots$, 
	where $+2\pi k_{2i+1}/p$ (\emph{resp.} $-2\pi k_{2i}/q$) is the value of the angle $\widehat{B^{2i}A^{2i+1}B^{2i+2}}$ (\emph{resp.} $\widehat{A^{2i-1}B^{2i}A^{2i+1}}$). 
  	The code of a path starting in the orbit of~$\Ao$ is defined similarly, but starts with a power of the letter~$b$. 
  	The code of a bi-infinite path $(\dots, A^{-1}, B^0, A^1, B^2, \dots)$ is the bi-infinite word $\dots a^{k_{-1}}b^{k_0}a^{k_1}b^{k_2}\dots$ defined in the same way.
\end{definition}

Beware the choice of the opposite orientations for $a$ and $b$. 
We denote by $\lex$ and $\slex$ the lexicographic ordering on words on the alphabet~$\{a, b\}$.
This choice ensures that the lexicographic ordering on codes and the clockwise ordering on~$\bord\Hy$  coincide (Lemma~\ref{L:Ordering}). 
Note that two paths obtained one from the other by an element of~$\Gpqr$ have the same code.
Following an edge corresponds to suppressing the first block of $a$ of $b$ in the code of a path. 
We then define the \emph{super-shift} operator~$\sshift$ on infinite words that suppresses, not only the first letter of a word but, the first block of similar letters of a word.

It is easy to see that every geodesic in~$\Hy$ is quasi-isometric to a path in~$\grpqr$, but the latter is not unique. 
The words~$u_L, u_R, v_L, v_R$, $w_L, w_R$ being given by Table~\ref{Table} and~2, the goal of this section is to prove Theorem~\ref{P:Coding} of the introduction which states that for (almost) every closed geodesic there is a way to choose a unique code.


\begin{figure}[ht]
  	\begin{picture}(120,120)(0,0)
      		\put(0,0){\includegraphics*[width=.8\textwidth]{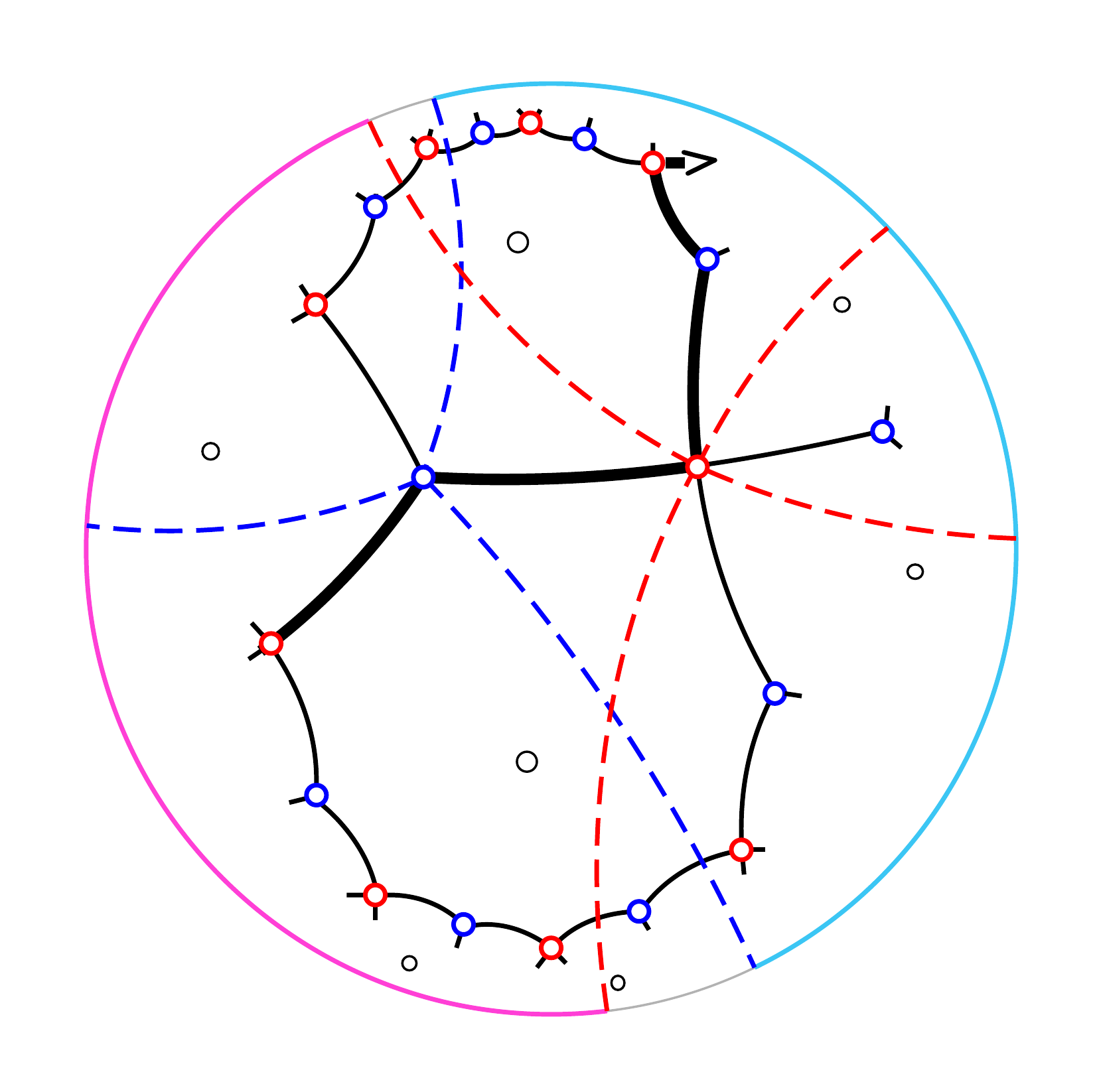}}
		\put(44,62){$\Ao$}
		\put(77,63.5){$\Bo$}
		\put(55,88.5){$\Co$}
		\put(112,68){$I_{\Ao\to\Bo}$}
		\put(-3,50){$I_{\Bo\to\Ao}$}
  	\end{picture}
  	\caption{\small The graph~$\grpqr$ in the case~$p=3, q=4, r=5$. In bold is a path through the graph whose code starts with $aba^2b^3$. 
	A choice of a base pair of spectacles $P=\{I_{\Ao\to\Bo}, I_{\Bo\to\Ao}\}\subset\partial\Hy$ is also depicted.}
  	\label{F:Grpqr}
\end{figure}

\subsection{Pairs of spectacles and admissible paths}
	\label{S:Codes}

\begin{definition}
	\label{D:Pair}
	A \emph{base pair of spectacles} is a pair of disjoint semi-open intervals $\{I_{\Ao\to\Bo}$, $I_{\Bo\to\Ao}\}$ in~$\bord\Hy$, 
	such that all geodesic rays starting at~$\Ao$ and reaching $I_{\Ao\to\Bo}$ span an angle~$2\pi/p$ from~$\Ao$,  and
	all geodesic rays starting at~$\Bo$ and reaching $I_{\Bo\to\Ao}$ span an angle~$2\pi/q$ from~$\Bo$.
	For every oriented edge~$(A,B)$ of~$\grpqr$, the \emph{associated spectacle} $I_{A\to B}$ is the image of $I_{\Ao\to\Bo}$ 
	by the unique element of $\Gpqr$ that maps $(\Ao, \Bo)$ onto $(A,B)$. 
	The spectacles associated to an oriented edge~$(B,A)$ is defined similarly.
\end{definition}

\begin{remark}
It will usually make sense to choose the base pair of spectacles so that the geodesic rays from $\Ao$ to $I_{\Ao\to\Bo}$ include the ray passing through $\Bo$, and this will be the case for our choice of spectacles in the next section. However, we do not require this in the definition. 
\end{remark}

Deciding whether the intervals are open on the left or on the right is not important, but in order to avoid heavy case specifications, we arbitrarily decide that all intervals on~$\bord\Hy$ are open on the right and closed on the left.

Since $I_{\Ao\to\Bo}$ spans an angle~$2\pi/p$ from~$\Ao$, the $p$ spectacles associated to the $p$ edges starting at a point in the orbit of~$\Ao$ tesselate~$\bord\Hy$ (and so do the $q$ spectacles associated to the edges starting at a point in the orbit of~$B_0$), for any choice of base pair of spectacles.

\begin{definition}
  	\label{D:Admissible}
  	Given a base pair of spectacles~$\Spec$, a semi-infinite or bi-infinite  path~$\gamma$ in~$\grpqr$ is \emph{$\Spec$-admissible} 
	if it has a limit in~$\bord\Hy$, say~$\xi$, and if $\xi$ belongs to all spectacles associated to all oriented edges of~$\gamma$.
\end{definition}

Given a pair of spectacles~$\Spec$ and a point~$\xi$ on~$\bord \Hy$, it is not clear whether there exists an $\Spec$-admissible semi-infinite path connecting~$\Ao$ (say) to~$\xi$. 
What is easy to check is that if there is such an admissible path, it is unique.

\begin{lemma}
  	\label{L:Uniqueness}
  	Given a pair of spectacles~$\Spec$, for every $A$ in~$\grpqr$ and every $\xi$ in~$\bord\Hy$, 
	there exists at most one $\Spec$-admissible path in~$\grpqr$ joining~$A$ to~$\xi$. 
\end{lemma} 

\begin{proof}
	The path can be constructed inductively: if the path begins by $(A^0, B^1, \dots, A^{2i})$ (\emph{resp.} $(A^0, B^1, \dots, B^{2i+1})$), 
	there is a unique oriented edge starting at~$A^{2i}$ (\emph{resp.} $B^{2i+1}$) whose associated spectacle contains~$\xi$. 
	This gives a unique choice for~$B^{2i+1}$ (\emph{resp.} $A^{2i+2}$). 
	Note that it is not clear that the path so defined converges to~$\xi$.
\end{proof}

A nice feature of the coding with two letters given by Definition~\ref{D:coding} is that the natural cyclic ordering on admissible paths starting at a given point is reflected in the lexicographic order of the codes.

\begin{lemma}
  	\label{L:Ordering}
  	Let $\Spec$ be a base pair of spectacles. 
  	Assume that $\gamma, \gamma'$ are two $\Spec$-admissible paths in~$\grpqr$ that begin with the same edge. 
  	Then $\gamma$ is to the left of~$\gamma'$ after they diverge if and only if $w(\gamma)\slex w(\gamma')$ holds.
\end{lemma}

\begin{proof}
 	The main observation is that, after taking any edge of~$\grpqr$, the cyclic ordering on the $p-1$ or $q-1$ possible next edges coincide with the lexicographic ordering. 
  	Indeed, since $a$ corresponds to a rotation of angle~$+\frac {2\pi} p$, 
	a path whose code begins with $ab$ is on the right of a path whose code begins with $a^2b$, and so on, and the leftmost paths have codes beginning with $a^{p-1}b$. 
  	Similarly, since $b$ corresponds to a rotation of angle~$-\frac {2\pi} q$, 
	a path whose code begins with $ba$ is on the left of a path whose code begins with $b^2a$, and so on.
  	Therefore, two paths coincide as long as their codes coincide. 
  	As soon as they diverge, the leftmost has the smallest code.
\end{proof}


\subsection{Accurate pairs of spectacles and existence of admissible paths}
  	\label{S:GoodSpectacles}

We have seen how a choice of a base pair of spectacles yields a coding for every path in~$\grpqr$, and leads to the notion of admissible infinite path. 
Now we want to find spectacles that ensure that every bi-infinite geodesic in~$\Hy$ can be shadowed by an admissible bi-infinite path.
This is the notion of~\emph{accurate spectacles}. 

\begin{definition}
	\label{D:Good}
	A base pair of spectacles~$\Spec$ is said to be~\emph{accurate} if for every~$A$ in~$\Hy$ and for every $\eta, \xi$ in~$\bord\Hy$ 
	there exist an $\Spec$-admissible semi-infinite path connecting $A$ to~$\xi$ and an $\Spec$-admissible bi-infinite path connecting~$\eta$ to~$\xi$. 
	These paths are denoted by~$\gamma^\Spec_{A\to\xi}$ and~$\gamma^\Spec_{\eta\to \xi}$ respectively.
\end{definition}

Let us see which codes are yielded by accurate pairs of spectacles.

\begin{definition}
	\label{D:KneadingSequences}
	Assume that~$\Spec=\{I_{\Ao\to\Bo}, I_{\Bo\to\Ao}\}$ is an accurate pair of spectacles. 
	Denote by $u^\Spec_L$ and $u^\Spec_R$ the respective codes of the admissible paths connecting~$\Ao$ to the left and right extremities of~$I_{\Ao\to\Bo}$, 
	and by $v^\Spec_L$ and $v^\Spec_R$ the codes of admissible the paths connecting~$\Bo$ to the two extremities of~$I_{\Bo\to\Ao}$. 
	The sequences~$u^\Spec_L, u^\Spec_R, v^\Spec_L, v^\Spec_R$ are called the \emph{kneading sequences} associated to the pair of spectacles~$\Spec$.
\end{definition}

\begin{lemma}
	\label{L:Admissible}
	Given an accurate pair of spectacles~$\Spec$, a semi-infinite word~$w=b^{k_{1}}a^{k_2}b^{k_{3}}\dots$ is the code of an $\Spec$-admissible path 
	if and only if it satisfies $u^\Spec_L \lex \sshift^{2i+1}(w) \slex u^\Spec_R$ and $v^\Spec_L \lex \sshift^{2i}(w) \slex v^\Spec_R$ for every~$i$.
\end{lemma}

\begin{proof}
	Let $\xi$ be a point~$\bord\Hy$. 
	By definition of accurate spectacles, there exists an $\Spec$-admissible path~$\gamma:=(A^0,B^1, A^2, B^3,\dots)$ 
	that connects~$\Ao$ to~$\xi$ and by Lemma~\ref{L:Uniqueness} this $\Spec$-admissible path is unique. 
	Denote its code by~$w$. 
	For every~$i$ the code of the path~$(A^{2i},B^{2i+1},\dots)$ is the word~$\sshift^{2i}(w)$. 
	Since $\xi$ belongs the the interval~$I_{A^{2i}\to B^{2i+1}}$, by Lemma~\ref{L:Ordering}, we have $v^\Spec_L \lex \sshift^{2i}(w) \slex v^\Spec_R$.
	Similarly, since $\xi$ belongs the the interval~$I_{B^{2i+1}\to A^{2i+2}}$, we have $u^\Spec_L \lex \sshift^{2i+1}(w) \slex u^\Spec_R$.

	Conversely, if~$w$ satisfies the above constraints, denote the associated path by $\gamma$ and the limit of~$\gamma$ in~$\bord\Hy$ by~$\xi$. 
	Lemma~\ref{L:Uniqueness} implies that $\xi$ belongs to all spectacles~$I_{A^{2i}\to B^{2i+1}}$ and~$I_{B^{2i+1}\to A^{2i+2}}$. 
	This means exactly that $w$ is $\Spec$-admissible.
\end{proof}

By the same argument, we immediately get

\begin{lemma}
	\label{L:BiAdmissible}
	Given an accurate pair of spectacles~$\Spec$, a bi-infinite word~$w=\dots b^{k_{-1}}a^{k_0}b^{k_{1}}\dots$ is the code of an $\Spec$-admissible path 
	if and only if it satisfies $u^\Spec_L \lex \sshift^{2i+1}(w) \slex u^\Spec_R$ and $v^\Spec_L \lex \sshift^{2i}(w) \slex v^\Spec_R$ for every~$i$.
\end{lemma}


\subsection{An explicit choice of spectacles}
\label{S:Explicit}
 
Accurate pairs of spectacles yield codes that are easy to describe. 
However, it is not obvious that there exist accurate pairs of spectacles.
For example, if the spectacles~$I_{\Ao\to\Bo}$ do not include the endpoint of the geodesic ray from $\Ao$ through $\Bo$, then an admissible path~$\gamma^\Spec_{A\to\xi}$ would veer away from~$\xi$ instead of converging toward~$\xi$. 
As another example, note that when $r$ is infinite, there exists a unique pair of spectacles. 
Indeed in this case the graph~$\grpqr$ is a tree, and the only possibility for the intervals~$I_{\Ao\to\Bo}$ and~$I_{\Bo\to\Ao}$ is to be the two intervals that connect the extremities of the two horoballs meeting on the edge~$(\Ao\Bo)$.

The goal of this section is to construct explicitly an accurate pair of spectacles in the case where $r$ is finite.

For the rest of this section, assume $p>2$.
The case $p=2$ will be treated in Section~\ref{S:p=2}.
Assume  $F$ is a region of the complement of~$\grpqr$ in~$\Hy$. 
Since $F$ has $2r$ vertices, it makes sense to say that two edges are \emph{opposite} in~$F$: if there are $r-1$~edges between them in both directions.

\begin{definition}[see Figures~\ref{F:bigon} and \ref{F:Borders}]
  	\label{D:BigonChain}
  	Assume that $A,B$ are two adjacent vertices of~$\grpqr$. 
  	Then the associated \emph{bigon}~$\bigon_{AB}$ is defined as the infinite sequence $(F^1, F^2, \dots)$ of faces of the complement of~$\grpqr$, 
	where $F^1$ is the face on the left of the edge~$(AB)$, the edge $F^1\cap F^2$ is opposite to $(AB)$ in $F^1$, 
	and $F^{i}\cap F^{i+1}$ is opposite to $F^{i-1}\cap F^{i}$ in $F^i$ for every~$i\ge 2$.

  	The bigon~$\bigon_{AB}$ has a left and a right boundary starting at $A$ and $B$ respectively.
  	They both converge to the same point in~$\bord\Hy$ that we call the \emph{normal extremity} of the oriented edge~$(AB)$ and denote by~$\xi_{AB}$.
\end{definition}

\begin{figure}[ht]
  	\begin{picture}(117,120)(0,0)
      		\put(0,0){\includegraphics*[width=.75\textwidth]{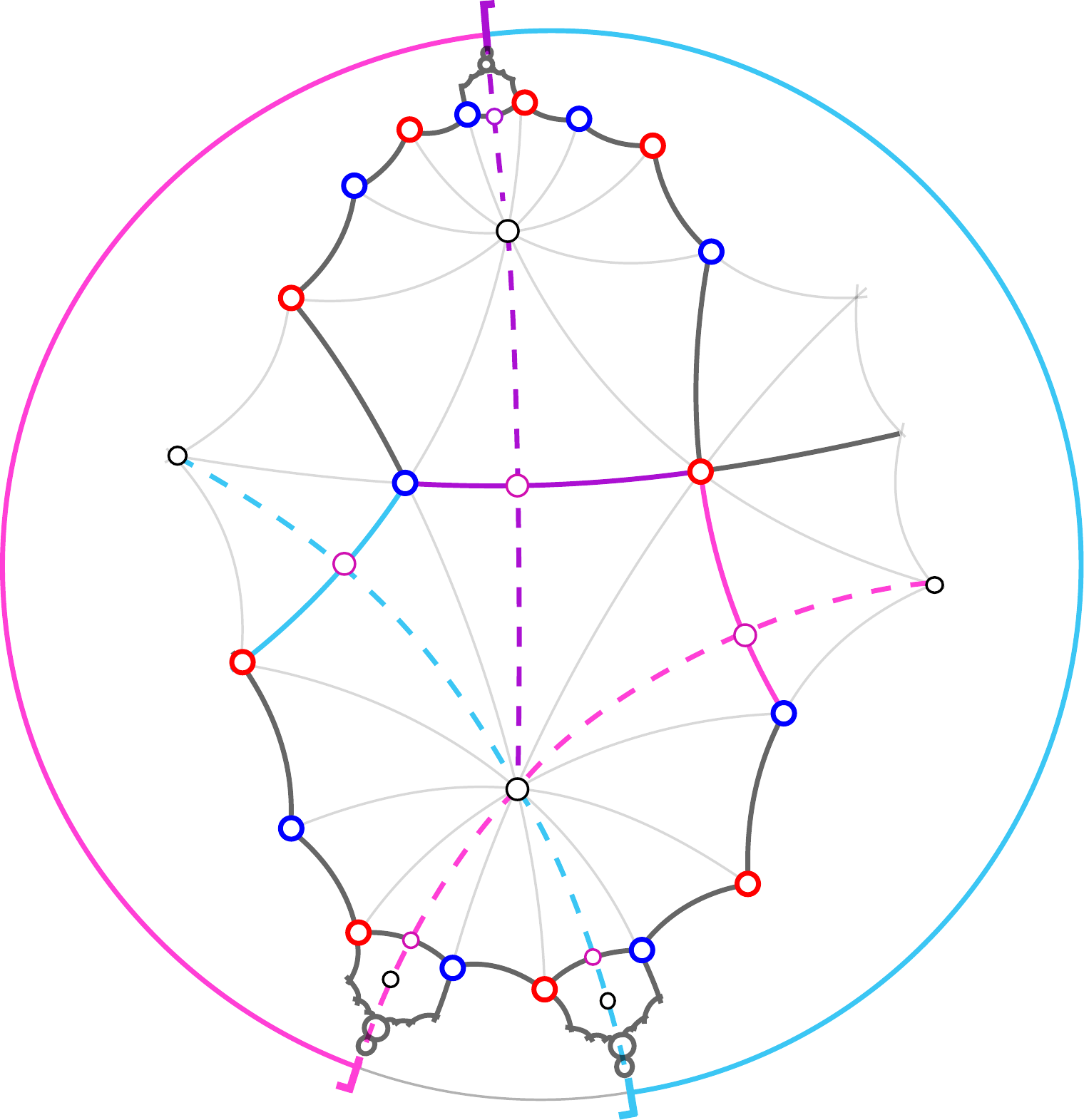}}
		\put(47,118){$\xi_{\Ao\Bo}$}
 		\put(41,62){$\Ao$}
		\put(74,63.5){$\Bo$}
		\put(50,88){$\Co$}
		\put(112,68){$I^f_{\Ao\to\Bo}$}
		\put(-12,50){$I^f_{\Bo\to\Ao}$}
 	\end{picture}
  	\caption{\small The  pair of spectacles~$\Spec^f$}
  	\label{F:bigon}
\end{figure}
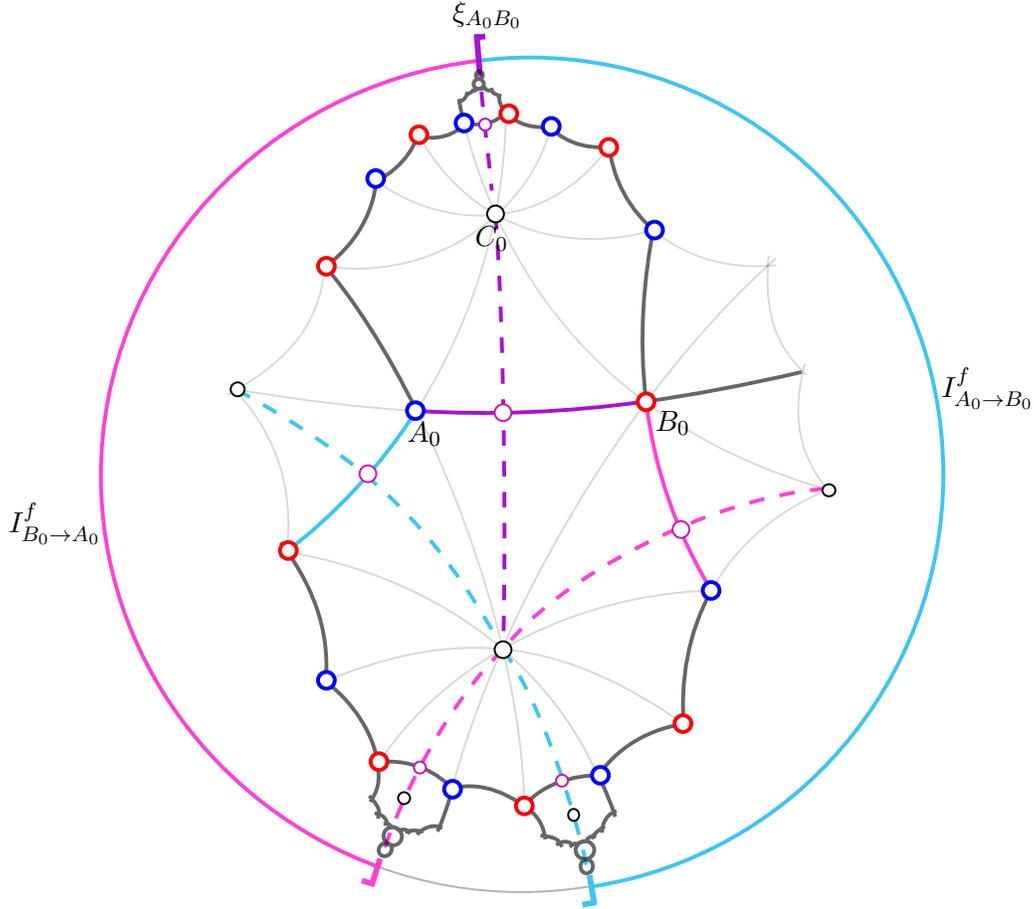

\begin{lemma}
  	\label{L:KneadingTpqr}
  	The codes of the paths that follow the left and right boundaries of~$\bigon_{AB}$ are the words~$v_R$ and $u_L$ given by Table~1.
\end{lemma}

\begin{proof}
	We have to draw the corresponding bigon, and to describe its borders. 
	This is depicted in Figure~\ref{F:Borders}.
	The code is periodic, the period corresponding to one face in the case of $r$ odd and two faces in the case of $r$ even. 
\end{proof}

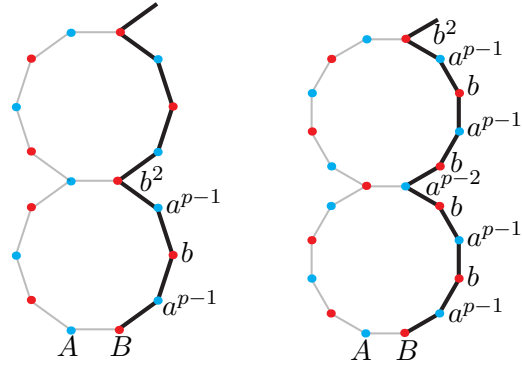
\begin{figure}[ht]
  	\begin{picture}(70,50)(0,0)
      		\put(0,0){\includegraphics*[width=.4\textwidth]{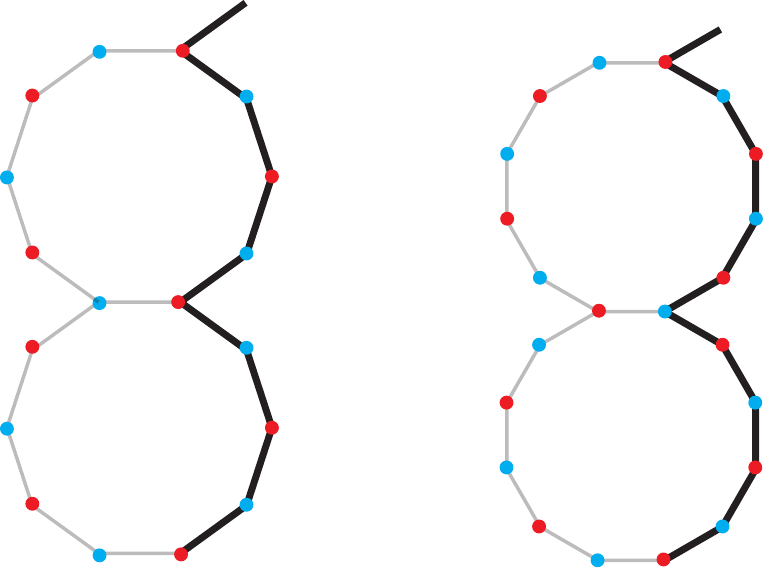}}
		\put(6,-2.5){$A$}
		\put(13,-2.5){$B$}
		\put(45,-3){$A$}
		\put(51,-3){$B$}
		\put(20,3){$a^{p-1}$}
		\put(22.5,10){$b$}
		\put(20.5,16.5){$a^{p-1}$}
		\put(17,19.5){$b^2$}
		\put(58,2){$a^{p-1}$}
		\put(60.5,6.79){$b$}
		\put(60.5,12.2){$a^{p-1}$}
		\put(58.3,16.2){$b$}
		\put(55,19){$a^{p-2}$}
		\put(58.3,22){$b$}
		\put(60.5,26.79){$a^{p-1}$}
		\put(60.5,32.2){$b$}
		\put(58,36.2){$a^{p-1}$}
		\put(56,39){$b^{2}$}
 	\end{picture}
  	\caption{\small The code $u_L$ of the right border of on infinite bigon is the periodic word $((a^{p-1}b)^{\frac{r-3}{2}} a^{p-1}b^2)^\infty$ 
	when $r$ is odd (on the left with $r=5$), and $((a^{p-1}b)^{\frac{r-2}{2}} a^{p-2}(ba^{p-1})^{\frac{r-2}{2}}b^2)^\infty$ when $r$ is even (on the right with $r=6$). }	\label{F:Borders}
\end{figure}

Let $\Spec^f = (I^f_{\Ao\to\Bo}, I^f_{\Bo\to\Ao})$ be the pair of spectacles such that the left extremity of~$I^f_{\Ao\to\Bo}$ and the right extremity of~$I^f_{\Bo\to\Ao}$ both coincide with the point~$\xi_{\Ao\Bo}$ (see Figure~\ref{F:bigon}).

\begin{lemma}
	\label{L:HalfPathsConverge}
	For every $A$ in~$\grpqr$ and $\xi$ on~$\bord\Hy$, the path~$\gamma^{\Spec^f}_{A\to \xi}$ converges to~$\xi$. 
\end{lemma}

\begin{proof}
	Let $(A^0, B^1, A^2, \dots)$ denote the consecutive vertices visited by~$\gamma^{\Spec^f}_{A\to \xi}$ and $b_\xi$ be a Busemann function~\cite{Busemann} on~$\Hy$ associated to~$\xi$ .
	We claim that~$b_\xi$ is decreasing when evaluated along every second point of~$\gamma^{\Spec^f}_{A\to \xi}$, 
	that is, $b_\xi(A^0) > b_\xi(A^2) > b_\xi(A^4) > \dots$ holds.
	Indeed, one checks (see Figure~\ref{F:Busemann}) that for every $i$ the spectacles~$I_{B^{2i+1}\to A^{2i+2}}$ are included 
	inside the interval of~$\bord\Hy$ consisting of those directions that are closer to~$A^{2i+2}$ than to~$A^{2i}$. \end{proof}

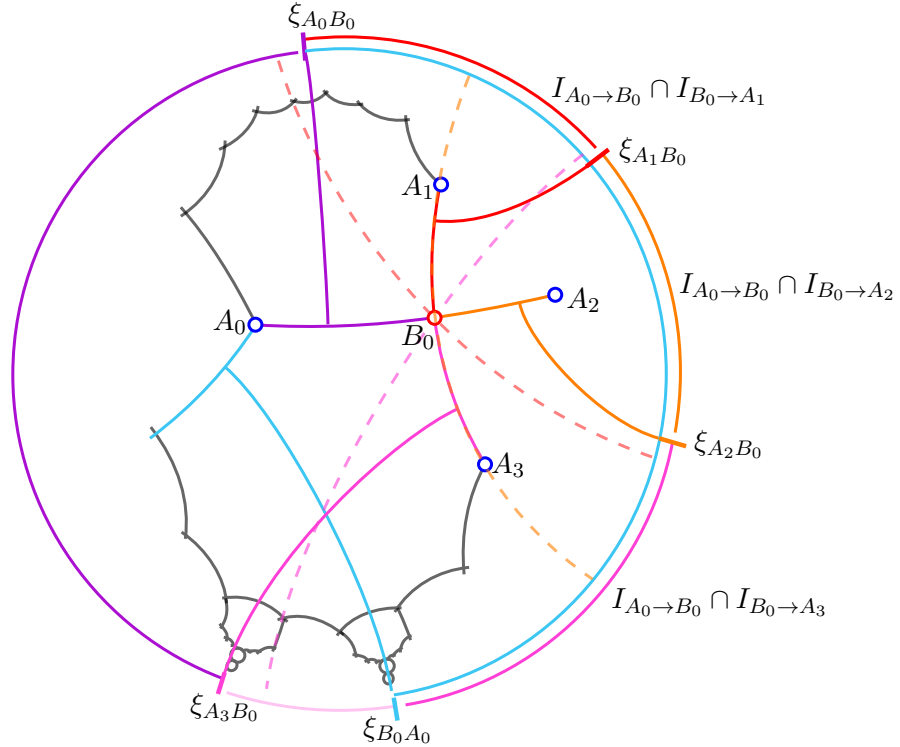
\begin{figure}[ht]
  	\begin{picture}(95,95)(0,0)
      		\put(0,0){\includegraphics*[width=.6\textwidth]{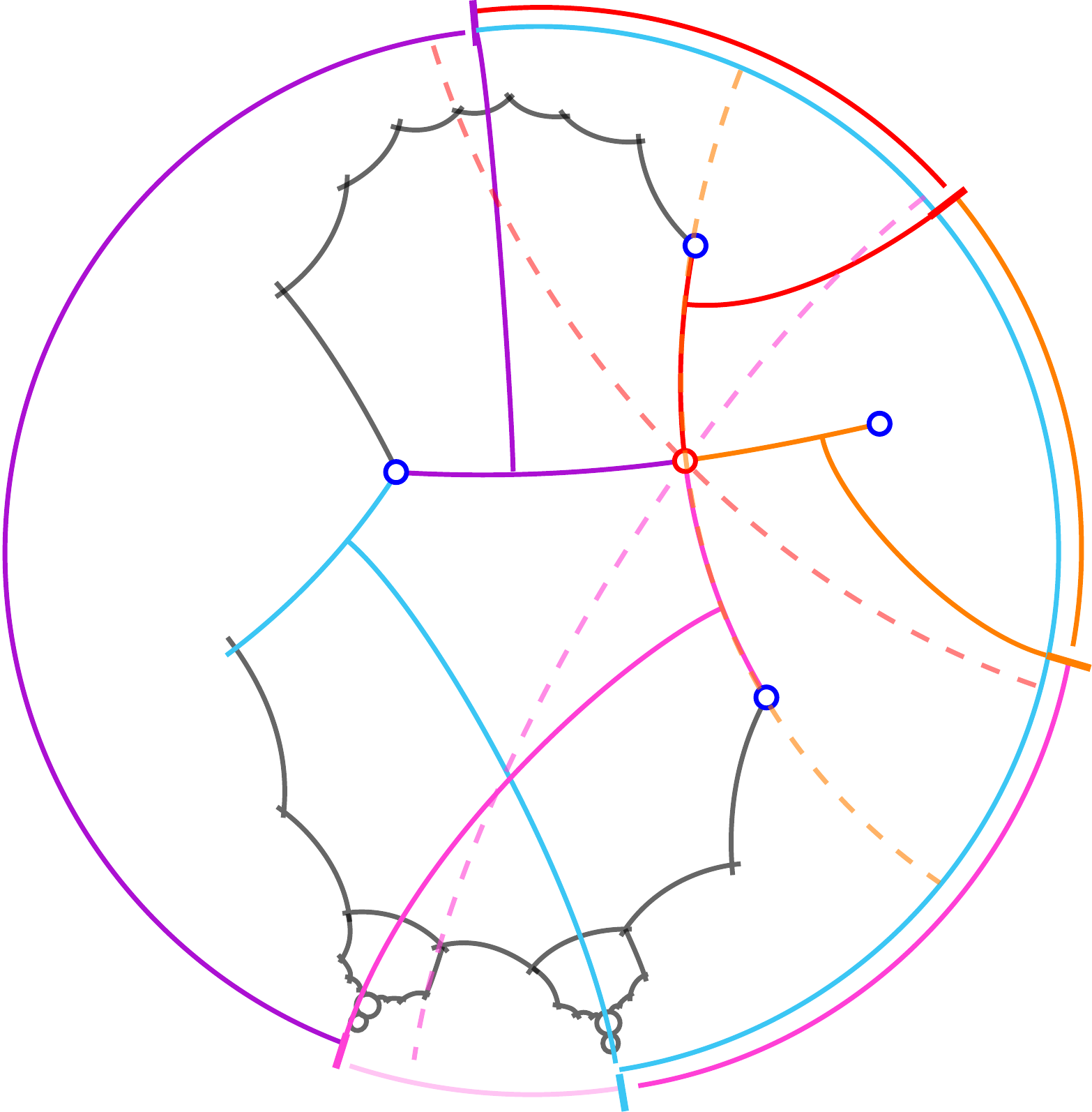}}
		\put(27,52){$A_0$}
		\put(52,70){$A_1$}
		\put(74,55){$A_2$}
		\put(64,33){$A_3$}
		\put(52,50){$B_0$}
		\put(72,83){$I_{\Ao\to\Bo}\cap I_{\Bo\to A_1}$}
		\put(89,57){$I_{\Ao\to\Bo}\cap I_{\Bo\to A_2}$}
		\put(80,15){$I_{\Ao\to\Bo}\cap I_{\Bo\to A_3}$}
		\put(37,93){$\xi_{\Ao\Bo}$}
		\put(81,75){$\xi_{A_1\Bo}$}
		\put(91,36){$\xi_{A_2\Bo}$}
		\put(24,1){$\xi_{A_3\Bo}$}
		\put(47,-2){$\xi_{\Bo\Ao}$}
 	\end{picture}
  	\caption{\small For every~$\xi$ in the interval~$I_{\Ao\to\Bo}\cap I_{\Bo\to A_1}$ (approximately the interval between 12:00 and 1:30), the Busemann function~$b_\xi$ is smaller at $A_1$ than at~$A_0$: 
	indeed $I_{\Ao\to\Bo}\cap I_{\Bo\to A_1}$ is included in the half space defined by the perpendicular bisector of~$(\Ao A_1)$ (dashed) and containing~$A_1$. 
	Similarly for $A_2$ (1:30 to 3:300) and $A_3$ (3:30 to 6:00).}
  	\label{F:Busemann}
\end{figure}

Now we have to see why for every $\eta,\xi$ on~$\bord\Hy$ there exists an admissible path that connects~$\eta$ to~$\xi$. 
 
\begin{lemma}
	\label{L:PathsConverge}
	The pair of spectacles~$\Spec^f$ is accurate. 
\end{lemma}

\begin{proof}
	We have to show that for every $\eta, \xi$ on~$\bord\Hy$ there exists an $\Spec^f$-admissible path connecting~$\eta$ to~$\xi$. 
	Consider a sequence~$(A_n)_{n\in\N}$ of points in the orbit of~$\Ao$ that converges to~$\eta$ and stays at a bounded distance 
	from the hyperbolic geodesic connecting~$\Ao$ to~$\eta$. 
	By Lemma~\ref{L:HalfPathsConverge} all paths~$\gamma^{\Spec^f}_{A_n\to \xi}$ converge to~$\xi$. 
	Also every path~$\gamma^{\Spec^f}_{A_n\to \xi}$ is at a bounded distance from the geodesic connecting~$A_n$ to~$\xi$.
	Therefore, every path~$\gamma^{\Spec^f}_{A_n\to \xi}$ is at a bounded distance from the geodesic that connects~$\eta$ to~$\xi$. 
	Since the orbit of~$\Ao$ contains only finitely many points in every ball of bounded radius, 
	this implies that infinitely many paths~$\gamma^{\Spec^f}_{A_n\to \xi}$ go through the same point, hence ultimately coincide. 
	By extraction, this yields an $\Spec^f$-admissible path connecting~$\eta$ to~$\xi$.
\end{proof}


\subsection{Uniqueness of the coding}
\label{S:Uniqueness}

The last missing point for proving Theorem~\ref{P:Coding} in the case $p>2$ is to see when admissible paths fail to be unique. 
Denote the word $(a^{p-1}b)^{\frac{r-3}{2}} a^{p-2}b$ when $r$ is odd and $(a^{p-1}b)^{\frac{r-2}{2}} a^{p-2}b(a^{p-1}b)^{\frac{r-4}{2}}a^{p-2}b$ when $r$ is even by $x_L$, and the word $(b^{q-1}a)^{\frac{r-3}{2}} b^{q-2}a$ when $r$ is odd and $(b^{q-1}a)^{\frac{r-2}{2}} b^{q-2}a(b^{q-1}a)^{\frac{r-4}{2}}b^{q-2}a$ when $r$ is even by $x_R$.
We then have (see Table~2) $w_L=x_L^\infty$ and $w_R=x_R^\infty$.
	
\begin{lemma}
	\label{L:BiUniqueness}
	Assume $r$ finite.
	Suppose that $\gamma_1, \gamma_2$ are two $\Spec^f$-admissible paths connecting the same points on~$\bord\Hy$. 
	Then either $\gamma_1, \gamma_2$ coincide in a neighbourhood of~$\xi$ in which case 
	their codes are of the form~$(\dots x_Lx_Lx_L)w_1$ and~$(\dots x_Rx_Rx_R)w_2$ 
	where $w_1$ and $w_2$ are semi-infinite words that ultimately coincide, or $\gamma, \gamma'$ are disjoint in which case their codes are $x_L^\infty$ and $x_R^\infty$.
\end{lemma}

\begin{proof}
	First suppose that~$\gamma_1, \gamma_2$ have no point in common. 
	Since $\gamma_1, \gamma_2$ have the same extremities on~$\bord\Hy$, 
	they are separated by an infinite strip of faces of~$\Hy\setminus\grpqr$, namely there exists faces~$\dots, F^{-1}, F^0, F^1, \dots$, 
	such that $F^{i+1}$ is adjacent to~$F^i$, that $\gamma_1$ is on the left of $\dots\cup F^{-1}\cup F^0\cup F^1\cup \dots$ and $\gamma_2$ on the right. 
	Note that the width of the strip cannot exceed one face.
	Indeed, since all faces are $2r$-gons with $r\ge 4$ whose vertices have angle at least $\pi/3$, one checks that in order to converge to the same point at infinity, 
	one of the paths would have to use more than half of the sides of a face. 
	This would contradict the admissibility of the path: indeed the code of such a path would contain at least $r/2$ consecutive blocks of the form~$a^{p-1}b$ (or $ab^{q-1}$), and this is prohibited by the kneading sequences of Table~\ref{Table}. 
	
	Now, for every~$i$, the intersection $F^i\cap F^{i+1}$ contains two vertices of~$\grpqr$ that we denote by $A^i$ and $B^i$. 
	Then $\gamma_1$ and $\gamma_2$ are two $\Spec^f$-admissible paths connecting~$B^i$ and $A^i$ to~$\xi$. 
	Since $\gamma_1$ does not go through $A^i$ and $\gamma_2$ not through $B^i$, 
	the point $\xi$ belongs to~$\bord\Hy \setminus (I_{A^i\to B^i}\cup I_{B^i\to A^i})$ (the bottom gray interval on Figure~\ref{F:bigon}) for every~$i$.
	This forces $\gamma_1\cap F^i \cap F^{i+1}$ to be in the orbit of~$A$ for every~$i$ and $\gamma_2\cap F^i \cap F^{i+1}$ to be in the orbit of~$B$. 

	Assume first $r$ is odd. 
	The previous remark implies that $\gamma_1$ and $\gamma_2$ both travel along $r-1$ sides of~$F^i$, so that the faces~$F^{i-1}$ and $F^{i+1}$ are opposite
	with respect to~$F^i$. Therefore, the codes of~$\gamma_1$ and $\gamma_2$ are $x_L^\infty$ and $x_R^\infty$.

	Next, assume $r$ is even. 
	The remark implies that $\gamma_1$ and $\gamma_2$ travel one along $r-2$ sides of~$F^i$ and the other along $r$ sides. 
	Similarly, in the face~$F^{i+1}$, the two paths travel along $r-2$ and $r$ sides of~$F^{i+1}$. 
	If the same path, say~$\gamma_1$, travels along $r$ sides of~$F^i$ and $F^{i+1}$, 
	then $\xi$ actually belongs to the interval~$I_{A^i\to B^i}$ so that $\gamma_1$ should have visited~$B^i$, a contradiction. 
	Therefore $\gamma_1$ must travel along $r$ sides of $F^i$ and $r-2$ of~$F^{i+1}$, 
	while~$\gamma_2$ travels along $r-2$ sides of~$F^i$ and $r$ sides of~$F^{i+1}$. 
	By induction, the codes of~$\gamma_1$ and $\gamma_2$ are $x_L^\infty$ and~$x_R^\infty$.

	Now if~$\gamma_1, \gamma_2$ have one point in common, by Lemma~\ref{L:Uniqueness}, they coincide after that point. 
	By the same argument as above, the part between them is a semi-infinite chain of~$2r$-gons. 
	The only possibility for such a chain is to be of the above form. 
\end{proof}

We can now conclude.

\begin{proof}[Proof of Theorem~\ref{P:Coding} in the case $p>2$]
	We consider the pair of spectacles~$\Spec^f$ introduced in Section~\ref{S:Explicit}.
	By Lemma~\ref{L:PathsConverge} the pair~$\Spec^f$ is accurate, meaning that for every $\eta, \xi$ on~$\bord\Hy$, 
	there exists an $\Spec^f$-admissible path connecting~$\eta$ to $\xi$. 
	By Lemma~\ref{L:BiAdmissible}, accurate paths are exactly the paths whose code verify the given inequalities. 
	Finally, by Lemma~\ref{L:BiUniqueness}, if $\eta, \xi$ are lifts of the extremities of a periodic geodesic this $\Spec^f$-admissible path is unique, 
	except in the case of the words $w_L$ and $w_R$. 
\end{proof}

\subsection{Case $p=2$} 
	\label{S:p=2}

All definitions and lemmas of Section~\ref{S:Codes} and~\ref{S:GoodSpectacles} and still valid when $p=2$. 
Since faces of~$\Hy\setminus\grpqr$ now have $r$ sides and not~$2r$, 
the definition of bigons (Definition~\ref{D:BigonChain}) and the description of their boundaries (Lemma~\ref{L:KneadingTpqr}) need to be adapted.

We still consider the faces as $2r$-gons (with angles at some of the vertices equal to $\pi$) and describe the corresponding boundaries. 
Up to switching between $q$ and $r$, we can suppose $r\ge 5$. 
The kneading sequences in this case are given by Table~\ref{Table} as shown by Figure~\ref{F:p=2}.

\begin{figure}[ht]
  	\begin{picture}(135,60)(0,0)
      		\put(0,-2){\includegraphics*[width=.865\textwidth]{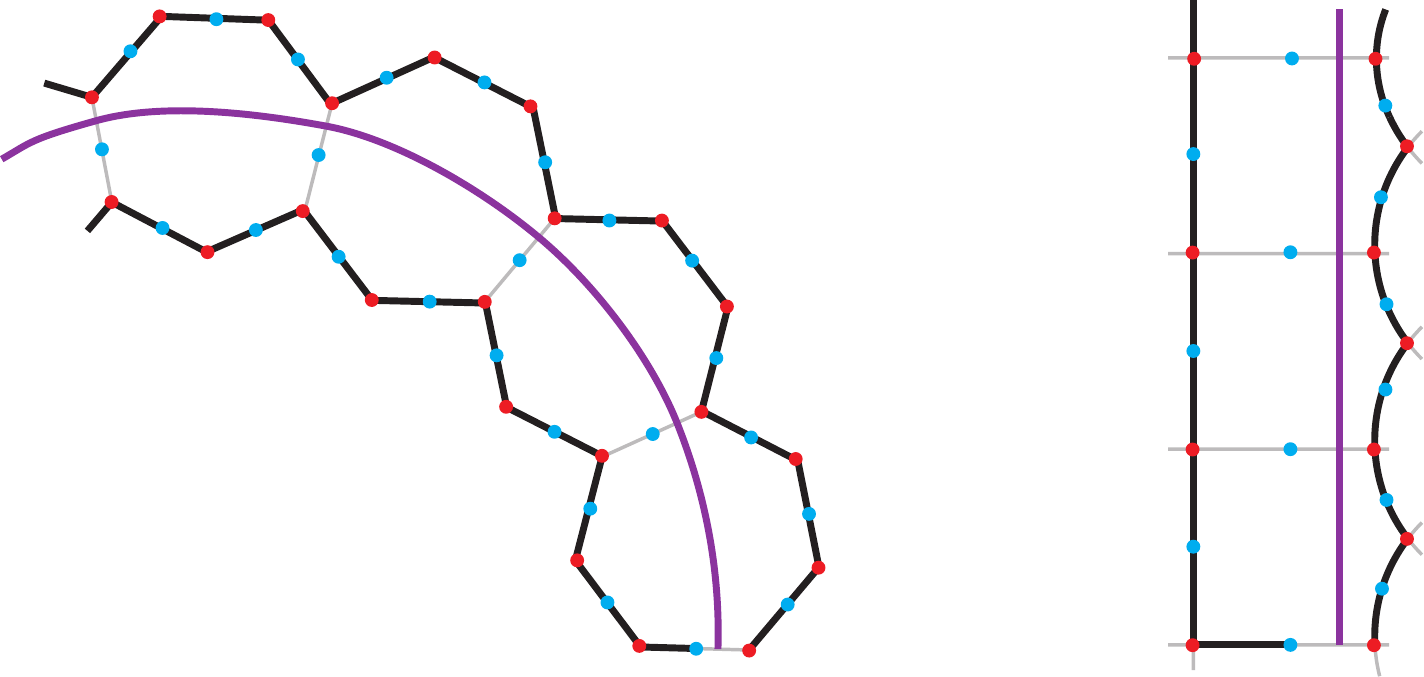}}
		\put(50,-1){$b^{q-1}$}
		\put(52,4){$a$}
		\put(45,7){$b^{q-1}$}
		\put(50.5,13){$a$}
		\put(46.5,15.5){$b^{q-2}$}
		\put(72.3,3.5){$a$}
		\put(75.5,7){$b$}
		\put(74.5,12.2){$a$}
		\put(73.6,16.7){$b$}
		\put(69.3,20){$a$}
		\put(65,22.2){$b^{2}$}
		\put(66,26.5){$a$}
		\put(67,31){$b$}
		\put(64,35.5){$a$}
		\put(61.5,39){$b$}
		\put(101,0){$b^{q-1}$}
		\put(105,9){$a$}
		\put(101,18){$b^{q-2}$}
		\put(105.5,27){$a$}
		\put(101,36){$b^{q-2}$}
		\put(127,5){$a$}
		\put(129.5,9){$b$}
		\put(128,13.5){$a$}
		\put(127,18){$b^2$}
		\put(127.5,23){$a$}
		\put(129.5,27){$b$}
		\put(127.5,31.5){$a$}
		\put(127,36){$b^2$}
 	\end{picture}
  	\caption{\small An infinite bigon in the case $p=2, r=7$ on the left. 
	The code $u_L$ of the right border is the periodic word $((ab)^{\frac{r-3}{2}} ab^2)^\infty$ when $r$ is odd (on the picture with $r=7$) 
	and  $((ab)^{\frac{r-4}{2}} ab^2)^\infty$ when $r$ is even. 
	The code $v_R$ of the left border is $b^{q-1}((ab^{q-1})^{\frac{r-5}{2}} ab^{q-2})^\infty$ when $r$ is odd 
	and $b^{q-1}((ab^{q-1})^{\frac{r-4}{2}} ab^{q-2})^\infty$ when $r$ is~even.
	The case $p=2, r=5$ on the right: the term $(ab^{q-1})^{\frac{r-5}{2}}$ disappears in $v_R$.}
  	\label{F:p=2}
\end{figure}
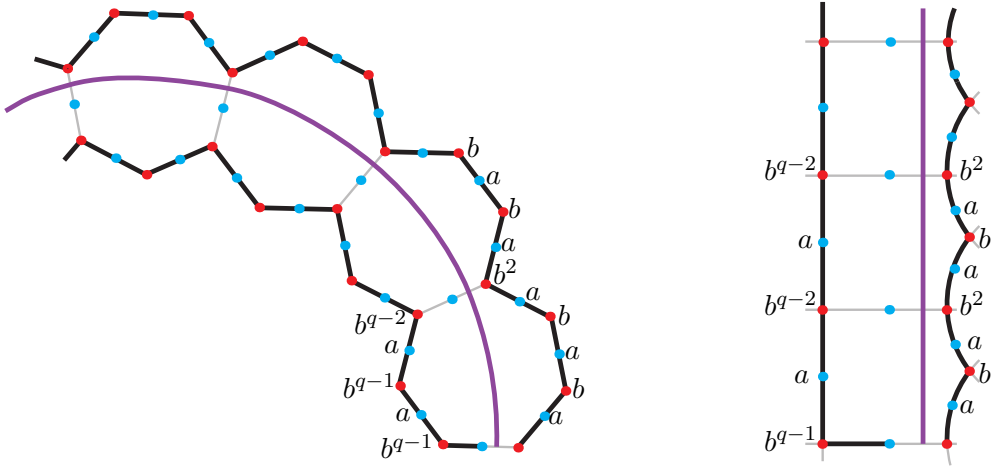

The proofs of Lemmas~\ref{L:HalfPathsConverge} and \ref{L:PathsConverge} need no special adaptation. 
Finally the proof of Lemma~\ref{L:BiUniqueness} can be translated to this case.
The only modification is that for every $i$ the intersection $F^i\cap F^{i+1}$ consists of two vertices $B_1^i$ and $B_2^i$ that are both adjacent to some point~$A^i$,
and thus the point $\xi$ belongs to $\bord\Hy \setminus (I_{A^i\to B_1^i} \cup I_{A^i\to B_2^i})$. 
If $r$ is odd this forces $\gamma_1$ to travel along $\frac{r-3}2$ sides of $F^i$ and $\gamma_2$ along $\frac{r-1}2$ sides (or $\frac{r-1}2$ and $\frac{r-3}2$ sides respectively). 
At the next face~$F^{i+1}$, the same argument as in the proof of Lemma~\ref{L:BiUniqueness} proves that the numbers have to alternate. 
Therefore the codes of $\gamma_1$ and $\gamma_2$ are those given by Table~\ref{TableExc}.
If $r$ is even the same proof also gives the codes of Table~\ref{TableExc}.

Finally the proof of Theorem~\ref{P:Coding} is now exactly the same that in the case~$p>2$.


\section{The template $\Tpqr$}
	\label{S:Tpqr}

We now introduce the main character of this story: the template~$\Tpqr$ in~$\utS$. 
The basic idea is that Theorem~\ref{P:Coding} gives a canonical way to distort~$\utS$ onto the 2-complex that is made of the fibers of the points of~$\grpqr/\Gpqr$.
However, this 2-complex is not a template because it branches along two circles, namely the fibers of the two vertices, instead of along intervals. 
The trick for proving Theorem~\ref{T:Template} is then to replace the vertices of $\grpqr$ by \emph{roundabouts}, thus obtaining a template that has the desired properties.
More precisely, we distort any path in $\grpqr$ in a neighbourhood of each vertex. 
The crucial point is that since we make the distortion in $\Hy$, we can choose to push the path to one side of the vertex or the other as we wish. 
We are therefore able to arrange all paths that pass through a vertex so that they turn around the vertex in a fixed direction (counterclockwise for any image of $\Ao$, clockwise for any image of $\Bo$). 
Thanks to the admissibility of the original path, the transformed path will rotate by an angle at most $\frac{p-1}p2\pi$ around~$\Ao$, and at most $\frac{q-1}q2\pi$ around~$\Bo$. 

For $A, B$ two adjacent vertices of~$\grpqr$, we denote by $B^+$ the image of $B$ by a $+2\pi/p$-rotation around~$A$ and by $A^+$ the image of $A$ by a $-2\pi/q$-rotation around~$B$.

 \begin{figure}[ht]
   	\begin{picture}(110,110)(0,0)
        \put(0,0){\includegraphics*[width=.72\textwidth]{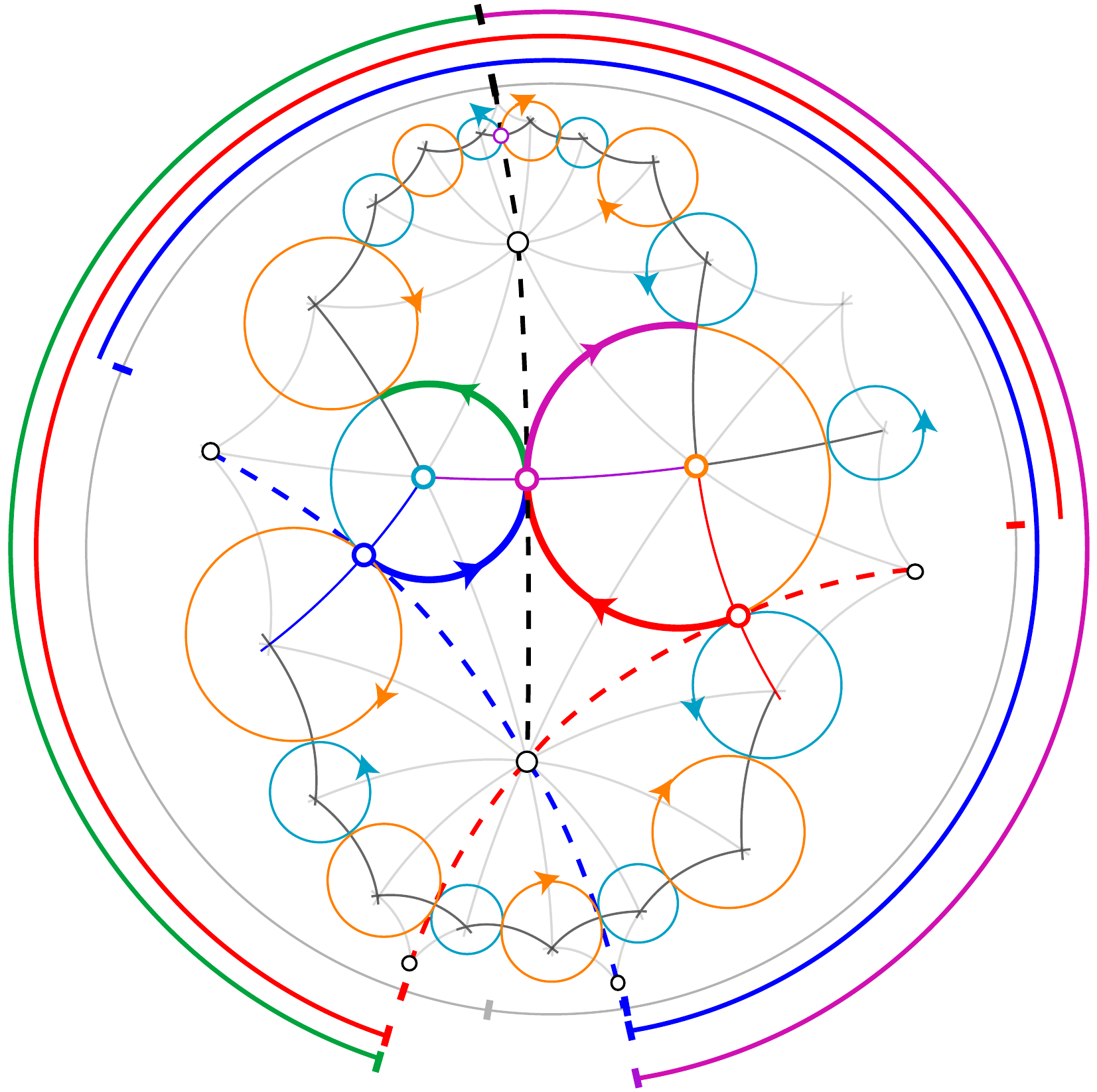}}
        \put(40,56){$A$}
        \put(65,57){$B$}
        \put(68,78){$A^+$}
        \put(27,77){$B^+$}
        \put(76,37){$A^-$}
        \put(22,44){$B^-$}
        \put(42,67){$\cerA$}
        \put(57,70){$\cerB$}
        \put(43,48){$\cerAm$}
        \put(57,50){$\cerBm$}
   	\end{picture}
   	\caption{\small The graph $\bouquetHy$.}
   	\label{F:Bouquet}
 \end{figure}

\begin{definition}[see Figure~\ref{F:Bouquet}]
  	\label{D:snake}
  	Let~$\bouquet$ be the embedded oriented graph in~$\Spqr$ which is a bouquet of two oriented circles, 
	one winding once around~$\Ad$ and one winding minus one times around~$\Bd$. 
  	Let~$\bouquetHy$ denote the lift of~$\bouquet$ in~$\Hy$. 
	Every lift of a circle is called a \emph{roundabout}. 
	A \emph{switch} is the contact point of two adjacent roundabouts.
  	For $A, B$ two adjacent vertices of~$\grpqr$, we denote by $\cerA$ the arc of the roundabout around~$A$ that connects the edge $(AB)$ to~$(AB^+)$ 
	and by $\cerB$ the arc of the roundabout around~$B$ that connects the edge $(AB)$ to~$(A^+B)$. 
\end{definition}

Every vertex of~$\grpqr$ is canonically associated to a roundabout of~$\bouquetHy$, and every edge~$\grpqr$ to a switch of~$\bouquetHy$.

\begin{definition}[see Figure~\ref{F:Tango}]
  	\label{D:tango}
  	Assume that $\gamma=(\dots, A^{-1}, B^0, A^1, \dots)$ is a simple path in~$\grpqr$.
  	The associated \emph{tango path} $\curvy\gamma$ is the path in~$\bouquetHy$ that is obtained by rotating around the roundabouts associated 
	to the vertices of $\gamma$ and changing at the switches that correspond to the edges of~$\gamma$.
\end{definition}

 \begin{figure}[ht]
 	\includegraphics*[width=.5\textwidth]{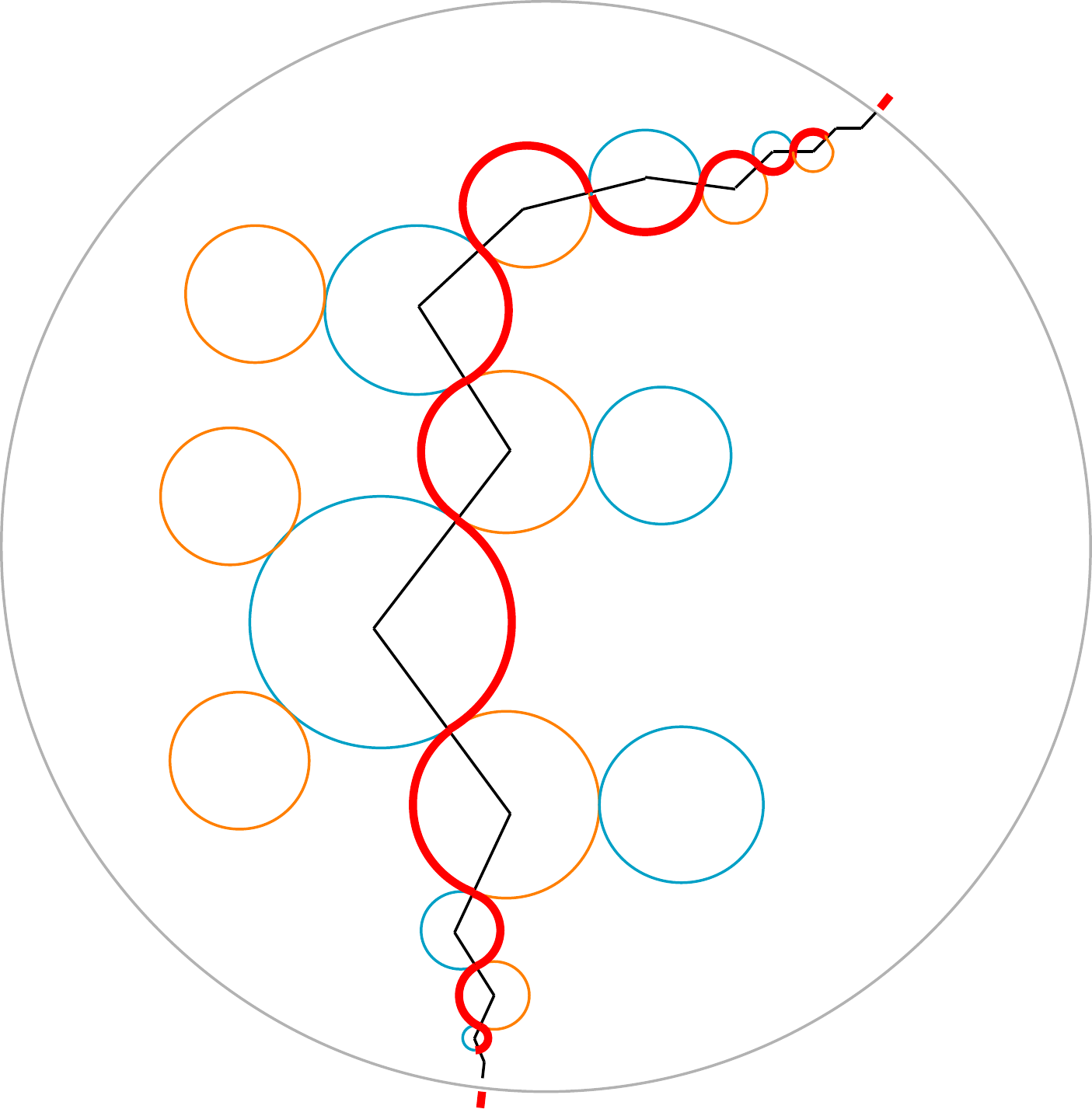}
   	\caption{\small The broken line represents a bi-infinite path in~$\grpqr$. The associated tango path in~$\bouquetHy$ is thickened. All circles correspond to roundabouts.}
   	\label{F:Tango}
 \end{figure}

Define $\tTpqr$ as the branched surface embedded in~$\ut\Hy$ that is made of those vectors of the form $(x,v)$, where $x$ is a point on~$\bouquetHy$ that belongs to an arc of the form $\cerA$ (\emph{resp.} $\cerB$) and $v$ is a tangent vector that points into $I^f_{B\to A}$ (\emph{resp.} $I^f_{A\to B}$).
The surface~$\tTpqr$ is a surface made of ribbons of the form $\cerA \times I^f_{B\to A}$ (\emph{resp.} $\cerB \times I^f_{A\to B}$) that branches in the fibers of the switches of~$\bouquetHy$.
We equip $\tTpqr$ with a semi-flow, denoted by~$\fttemp$, whose orbits are, on every ribbon, of the form $\cerA \times \{*\}$ (\emph{resp.} $\cerB \times \{*\}$). 
Since we are interested in the topology of orbits and not in the time, the speed is not relevant, we can arbitrarily decide that~$\fttemp$ travels along each branch at unit speed. 
In order to say that $\tTpqr$ is a template, we need to check that ribbons are glued nicely.

\begin{lemma}
  	\label{L:indeedtemplate}
  	Every branching arc of the surface $\tTpqr$ is of the form~$M\times I$, where $M$ is a switch of the graph~$\bouquetHy$ and, 
	denoting the vertices of~$\bouquetHy$ that surround~$M$ by $A, B$, the arc $I$ in~$\bord\Hy$ is the segment~$I^f_{A\to B} \cup I^f_{B\to A}$.
\end{lemma}

\begin{proof}
  	In the fiber of $M$ four branches meet, namely $\cerA, \cerB, \cerAm$, and $\cerBm$.
  	Their visual intervals are respectively $I^f_{B\to A}, I^f_{A\to B}, I^f_{B^-\to A}$, and $I^f_{A^-\to B}$. 
  	Figure~\ref{F:Bouquet} shows that all these segments are included in~$I^f_{A\to B} \cup I^f_{B\to A}$. 
\end{proof}

Lemma~\ref{L:indeedtemplate} ensures that~$\tTpqr$ is indeed a template: it is made of ribbons, its branching arcs are segments, each branching segment has two incoming ribbons ($\cerAm\times I^f_{B^-\to A}$ and  $\cerBm\times I^f_{A^-\to B}$) that overlap and two outgoing ribbons ($\cerA\times I^f_{B\to A}$ and  $\cerB\times I^f_{A\to B}$) that do not overlap.
Along the branching segments, the vector fields induced by the different ribbons coincide.
By construction, $\tTpqr$ is $\Gpqr$-equivariant, so we can mod out.

\begin{definition}
  	\label{D:Tpqr}
  	The template~$\Tpqr$ in~$\utS$ is defined as the quotient~$\tTpqr/\Gpqr$. 
  	It is equipped with the semi-flow $\ftemp$ that is the projection of $\fttemp$.
\end{definition}

We can now conclude.

\begin{proof}[Proof of Theorem~\ref{T:Template}]
  	First we check that $\Tpqr$ is embedded in the way described by Figure~\ref{F:Tpqr}. 
  	Indeed, contracting the ribbons of $\Tpqr$ into arcs bring $\Tpqr$ in a neighbourhood of the vector field~$v$ describes on Figure~\ref{F:PairOfPants}, 
	and the arcs $\cerA$, $\cerB$ wind once around the corresponding points $\Ad$ and $\Bd$. 
  	For the framing of the ribbons of~$\Tpqr$, notice that it is given by the direction of the fibers of~$\utS$. 
  	The framing of Figure~\ref{F:Tpqr} is obtained by a rotation of $90$ degrees, that is, by an isotopy. 
  	This proves that the template $\Tpqr$ is indeed embedded as in Figure~\ref{F:Tpqr}.

  	Now, all we have to do is to describe a map, say~$\deformation$, from~$\utS$ to~$\Tpqr$ that will transport the orbits of the geodesic flow~$\fgeod$ 
	onto the orbits of the semi-flow~$\ftemp$ on the template, and check that its restriction to finite collection of periodic orbits is a topological equivalence and an isotopy.
  	Actually, it is easier to do that in a $\Gpqr$-equivariant way in~$\ut\Hy$.
 
  	Assume that $\eta, \xi$ are two points in~$\bord\Hy$ and denote by~$g_{\eta\to\xi}$ the geodesics in~$\Hy$ that connects them.
  	By Theorem~\ref{P:Coding}, there also exists a unique admissible path~$\gamma^f_{\eta\to\xi}$ in~$\grpqr$ that connects $\eta$ and $\xi$.
  	We denote it by~${\curvy\gamma}_{\eta\to \xi}$ the associated tango path in~$\bouquetHy$. 
  	The lift of~$g_{\eta\to\xi}$ in~$\ut\Hy$ is the orbit~$g_{\eta\to\xi}\times\{\xi\}$ of the geodesic flow~$\fgeod$.
  	We now define~$\tdeformation$ as the map that takes~$g_{\eta\to\xi}\times\{\xi\}$ onto~$\curvy\gamma_{\eta\to\xi}\times\{\xi\}$, which is an orbit of~$\fttemp$.
  	Of course, this can be done using an isotopy in the level~$\Hy\times\{\xi\}$ and in a $\Gpqr$-equivariant way.
  	By Theorem~\ref{P:Coding}, the image of the map~$\tdeformation$ is exactly the template~$\tTpqr$.
  	Modding out the map~$\tdeformation$ by~$\Gpqr$, we obtain the desired map~$\deformation$. 

  	In order to check that we have an isotopy when we restrict to finite collections of periodic orbits, 
	assume that $\gamma, \gamma'$ are two arbitrary lifts in~$\ut\Hy$ of two periodic orbits of~$\fgeod$. 
  	Then $\gamma, \gamma'$ cannot point in the same direction in~$\bord\Hy$, 
	for otherwise they would become arbitrarily close, which is not possible for two periodic orbits. 
  	Therefore, the deformation~$\tdeformation$ can be realized for these two orbits by two isotopies that live in different levels of~$\ut\Hy$. 
  	These can be extended to a global isotopy of~$\ut\Hy$.
  	The same idea works for arbitrary (but finitely) many periodic orbits.
\end{proof}

\section{Concluding remarks}
\label{S:Questions}

\subsection{Templates with two ribbons}

In this article, we have given a construction of a template for some particular geodesic flows. 
One can wonder about the optimality of the result, that is, whether it can be extended to other geodesic flows on other surfaces or 2-dimensional orbifolds. 
It was proven in~\cite{BW} that in this general case, there exists always a template, but with no control on the number of ribbons and the way they are embedded. 
Explicit constructions were given in~\cite{Pierre}, but they always yield templates with more than two ribbons. 
Actually we have

\begin{proposition}
	Assume that $G$ is a Fuchsian group such that $\Hy/G$ is not a sphere with three cone points. 
	Then there is no template with two ribbons that describes the isotopy classes of all periodic orbits of the geodesic flow on~$\ut\Hy/G$.
\end{proposition}

\begin{proof}
	The set of isotopy classes that can be represented by a template with two ribbons is a submonoid with two generators of~$\pi_1(\ut\Hy/G)$. 
	Since the fundamental groups of all Fuchsian groups different from triangles groups have rank larger than~$2$, they cannot equal a monoid with two generators.
\end{proof}

The natural task is then to wonder what are the next simplest templates and which groups they represent.

\subsection{Space of codings}

In Section~\ref{S:Coding} we have constructed a particular coding, relying on a particular pair of accurate spectacles. 
It would be interesting to understand which codings can be obtained in this way.

\begin{question}
	Given $p,q,r$, what is the set of accurate spectacles for coding the periodic geodesics of~$\Hy/\Gpqr$?
\end{question}

More generally, given a generating set of~$\Gpqr$, a coding of periodic geodesics on~$\Hy/\Gpqr$ is a language. 
What we did is to describe one particular such language. 
Another example is given in~\cite{Pfeiffer} (it is not clear to us whether this coding can be obtained using an accurate pair of spectacles).

\begin{question}
	Given a generating set of~$\Gpqr$, what is the set of those languages that encode the periodic geodesics of~$\Hy/\Gpqr$?
\end{question}


\bibliographystyle{siam}

\end{document}